\documentclass[12pt]{amsart}

\usepackage{amsmath,amssymb,amsthm}
\usepackage{color, tikz}
\usepackage{tabularx}
\usepackage{comment}
   \usepackage{epstopdf}
   \usepackage{diagbox}
\usepackage{lscape}
\usepackage{caption}
\usepackage{rotating}
\usepackage{enumitem}
\usepackage[normalem]{ulem}
\usepackage{enumitem}

\usetikzlibrary{decorations.pathreplacing}
\usepackage{rotating}
\usetikzlibrary{positioning,chains,fit,shapes,calc}
\usetikzlibrary{arrows.meta}
\usetikzlibrary{shapes.geometric}
\usetikzlibrary{shapes.misc}
\usetikzlibrary{automata,positioning}

\makeatletter
\newcommand{\addresseshere}{%
  \enddoc@text\let\enddoc@text\relax
}
\makeatother

\tikzset{cross/.style={cross out, draw=black, line width=.5mm, rotate = 30},
cross/.default={1pt}}

\usetikzlibrary{shadings}

\usetikzlibrary{shapes}

\newcounter{x}
\newcounter{y}
\newcounter{z}

\newcommand\xaxis{210}
\newcommand\yaxis{-30}
\newcommand\zaxis{90}

\newcommand\topside[3]{
  \fill[fill=yellow, draw=black,shift={(\xaxis:#1)},shift={(\yaxis:#2)},
  shift={(\zaxis:#3)}] (0,0) -- (30:1) -- (0,1) --(150:1)--(0,0);
}

\newcommand\leftside[3]{
  \fill[fill=red, draw=black,shift={(\xaxis:#1)},shift={(\yaxis:#2)},
  shift={(\zaxis:#3)}] (0,0) -- (0,-1) -- (210:1) --(150:1)--(0,0);
}

\newcommand\rightside[3]{
  \fill[fill=blue, draw=black,shift={(\xaxis:#1)},shift={(\yaxis:#2)},
  shift={(\zaxis:#3)}] (0,0) -- (30:1) -- (-30:1) --(0,-1)--(0,0);
}

\newcommand\cube[3]{
  \topside{#1}{#2}{#3} \leftside{#1}{#2}{#3} \rightside{#1}{#2}{#3}
}

\newcommand\planepartition[1]{
 \setcounter{x}{-1}
  \foreach \a in {#1} {
    \addtocounter{x}{1}
    \setcounter{y}{-1}
    \foreach \b in \a {
      \addtocounter{y}{1}
      \setcounter{z}{-1}
      \foreach \c in {0,...,\b} {
        \addtocounter{z}{1}
      \ifthenelse{\c=0}{\setcounter{z}{-1},\addtocounter{y}{0}}{
        \cube{\value{x}}{\value{y}}{\value{z}}}
      }
    }
  }
}

\usepackage[margin=1in]{geometry}

\newtheorem{theorem}{Theorem}[section]
\newtheorem{prop}[theorem]{Proposition}

\newtheorem{lemma}[theorem]{Lemma}
\newtheorem{corollary}[theorem]{Corollary}

\theoremstyle{definition}
\newtheorem{remark}[theorem]{Remark}
\newtheorem{definition}[theorem]{Definition}

\newtheorem{example}[theorem]{Example}
\newtheorem{question}[theorem]{Question}

\newcommand{\Thex}{$\textrm{T}_{\textrm{hex}}$}
\newcommand{\ov}[1]{\overline{{#1}}}

\newcommand{\todo}[1]{\par \noindent
  \framebox{\begin{minipage}[c]{0.95 \textwidth} TO DO:
      #1 \end{minipage}}\par}

\newcommand\commentout[1]{}

\definecolor{green}{RGB}{34, 139, 34}
\newcommand{\jnote}[1]{\par \noindent
  \framebox{\begin{minipage}[c]{0.95 \textwidth}\color{green} JULIE'S NOTE:
      #1 \color{black}\end{minipage}}\par}

\definecolor{purple}{RGB}{128, 0, 128}
\newcommand{\hnote}[1]{\par \noindent
  \framebox{\begin{minipage}[c]{0.95 \textwidth}\color{purple} HELEN'S NOTE:
      #1 \color{black}\end{minipage}}\par}

\definecolor{orangered}{RGB}{255,69,0}

\newcommand\purplesout{\bgroup\markoverwith{\textcolor{purple}{\rule[0.5ex]{2pt}{0.4pt}}}\ULon}

\title{Matching complexes of trees and applications of the matching tree algorithm}

\begin{document}

\author{Marija Jeli\'c Milutinovi\'c}
\address{Faculty of Mathematics\\
         University of Belgrade\\
         Belgrade, Serbia}
\email{marijaj@matf.bg.ac.rs}

\author{Helen Jenne}
\address{CNRS and Institut Denis Poisson\\
         Universit\'e de Tours and Universit\'e d'Orl\'eans\\
         Tours, France}
\email{helen.jenne@univ-tours.fr}

\author{Alex McDonough}
\address{Department of Mathematics\\
         Brown University\\
         Providence, RI 02912--9032}
\email{amcd@math.brown.edu}

\author{Julianne Vega}
\address{Department of Mathematics\\
         Kennesaw State University\\
         Marietta, GA 30060}
\email{jvega30@kennesaw.edu}

\begin{abstract}
A matching complex of a simple graph $G$ is a simplicial complex with faces given by the matchings of $G$. The topology of matching complexes is mysterious; there are few graphs for which the homotopy type is known. Marietti and Testa showed that matching complexes of forests are contractible or homotopy equivalent to a wedge of spheres. We study two specific families of trees. For caterpillar graphs, we give explicit formulas for the number of spheres in each dimension and for perfect binary trees we find a strict connectivity bound. We also use a tool from discrete Morse theory called the \textit{Matching Tree Algorithm} to study the connectivity of honeycomb graphs, partially answering a question raised by Jonsson.
\end{abstract}

\maketitle
\pagestyle{plain}

\section{Introduction}

Let $G$ be a simple graph without isolated vertices with vertex set $V(G)$ and edge set $E(G)$. 
The \emph{matching complex} of a graph $G$, denoted $M(G)$, is a simplicial complex with vertices given by the edges of $G$ and faces given by matchings contained in $G$, where a \emph{matching} is a collection of pairwise disjoint edges of $G$.  We denote an edge in $G$ as $\overline{e} \in E(G)$ and the corresponding vertex in the matching complex of $G$ as $e \in M(G)$; see Figure~\ref{fig:matchingcomp}. 
The study of matching complexes frequently makes use of the fact that
the matching complex of a graph $G$ is the same as the {\em independence  complex} of the {\em line graph} of $G$. The {\em independence complex} of a graph $G$, $\textrm{Ind}(G)$, is the simplicial complex with vertex set $V(G)$ and faces given by sets of pairwise non-adjacent vertices in $G$. The {\em line graph} $L(G)$ is the graph with vertex set $E(G)$ with two vertices in $L(G)$ adjacent if and only if the corresponding edges are incident in $G$, where two edges in a graph are \emph{incident} if they share a common vertex.

We study matching complexes of polygonal line tilings and certain families of trees; namely, \em{caterpillar graphs} and \em{perfect binary trees}. We will detail our contributions in Section~\ref{sec:contributions}, but first we put our results in context by giving a brief overview of the matching complex literature.


\begin{figure} 
\definecolor{rvwvcq}{rgb}{0.08235294117647059,0.396078431372549,0.7529411764705882}
\definecolor{wrwrwr}{rgb}{0.3803921568627451,0.3803921568627451,0.3803921568627451}
\begin{tikzpicture}[scale = 0.5, line cap=round,line join=round,x=1cm,y=1cm]
\clip(-10.361965289256208,-1.487145454545453) rectangle (7.838034710743815,6.23285454545454);
\fill[line width=2pt,color=rvwvcq,fill=rvwvcq,fill opacity=0.10000000149011612] (6,3) -- (2.8796347107438045,5.007654545454543) -- (0,3) -- cycle;
\fill[line width=2pt,color=rvwvcq,fill=rvwvcq,fill opacity=0.10000000149011612] (0,3) -- (2.0084347107438036,1.4260545454545446) -- (0,0) -- cycle;
\fill[line width=2pt,color=rvwvcq,fill=rvwvcq,fill opacity=0.10000000149011612] (6,3) -- (3.9686347107438054,1.4502545454545446) -- (6,0) -- cycle;
\draw [line width=2pt] (0,0)-- (0,3);
\draw [line width=2pt] (6,3)-- (0,3);
\draw [line width=2pt] (6,3)-- (2.8796347107438045,5.007654545454543);
\draw [line width=2pt] (0,3)-- (2.8796347107438045,5.007654545454543);
\draw [line width=2pt] (6,3)-- (6,0);
\draw [line width=2pt] (0,0)-- (6,0);
\draw [line width=2pt] (3.9686347107438054,1.4502545454545446)-- (6,3);
\draw [line width=2pt] (3.9686347107438054,1.4502545454545446)-- (6,0);
\draw [line width=2pt] (3.9686347107438054,1.4502545454545446)-- (2.0084347107438036,1.4260545454545446);
\draw [line width=2pt] (2.0084347107438036,1.4260545454545446)-- (0,3);
\draw [line width=2pt] (2.0084347107438036,1.4260545454545446)-- (0,0);
\draw [line width=2pt,color=rvwvcq] (6,3)-- (2.8796347107438045,5.007654545454543);
\draw [line width=2pt,color=rvwvcq] (2.8796347107438045,5.007654545454543)-- (0,3);
\draw [line width=2pt,color=rvwvcq] (0,3)-- (6,3);
\draw [line width=2pt,color=rvwvcq] (0,3)-- (2.0084347107438036,1.4260545454545446);
\draw [line width=2pt,color=rvwvcq] (2.0084347107438036,1.4260545454545446)-- (0,0);
\draw [line width=2pt,color=rvwvcq] (0,0)-- (0,3);
\draw [line width=2pt,color=rvwvcq] (6,3)-- (3.9686347107438054,1.4502545454545446);
\draw [line width=2pt,color=rvwvcq] (3.9686347107438054,1.4502545454545446)-- (6,0);
\draw [line width=2pt,color=rvwvcq] (6,0)-- (6,3);
\draw (2.4,6) node[anchor=north west] {$g$};
\draw (6,3.492854545454543) node[anchor=north west] {$f$};
\draw (6,0.4728545454545449) node[anchor=north west] {$d$};
\draw (3.5,1.5) node[anchor=north west] {$b$};
\draw (1.65,1.4) node[anchor=north west] {$e$};
\draw (-0.9,0.4728545454545449) node[anchor=north west] {$a$};
\draw (-0.9,3.492854545454543) node[anchor=north west] {$c$};
\draw [line width=2pt] (-9.681965289256203,2.9128545454545436)-- (-9.701965289256206,0.3528545454545422);
\draw [line width=2pt] (-9.701965289256206,0.3528545454545422)-- (-6.801965289256203,0.3528545454545422);
\draw [line width=2pt] (-6.801965289256203,0.3528545454545422)-- (-6.761965289256203,2.872854545454541);
\draw [line width=2pt] (-6.761965289256203,2.872854545454541)-- (-9.681965289256203,2.9128545454545436);
\draw [line width=2pt] (-6.761965289256203,2.872854545454541)-- (-3.6219652892562006,2.8528545454545395);
\draw [line width=2pt] (-3.6219652892562006,2.8528545454545395)-- (-3.6619652892562016,0.3528545454545422);
\draw [line width=2pt] (-3.6619652892562016,0.3528545454545422)-- (-6.801965289256203,0.3528545454545422);

\draw [fill=wrwrwr] (0,0) circle (2.5pt);
\draw [fill=rvwvcq] (0,3) circle (2.5pt);
\draw [fill=rvwvcq] (6,3) circle (2.5pt);
\draw [fill=rvwvcq] (2.8796347107438045,5.007654545454543) circle (2.5pt);
\draw [fill=rvwvcq] (6,0) circle (2.5pt);
\draw [fill=rvwvcq] (3.9686347107438054,1.4502545454545446) circle (2.5pt);
\draw [fill=rvwvcq] (2.0084347107438036,1.4260545454545446) circle (2.5pt);
\draw [fill=rvwvcq] (-9.681965289256203,2.9128545454545436) circle (2.5pt);
\draw [fill=rvwvcq] (-9.701965289256206,0.3528545454545422) circle (2.5pt);
\draw [fill=rvwvcq] (-6.801965289256203,0.3528545454545422) circle (2.5pt);
\draw [fill=rvwvcq] (-6.761965289256203,2.872854545454541) circle (2.5pt);
\draw [fill=rvwvcq] (-3.6219652892562006,2.8528545454545395) circle (2.5pt);
\draw [fill=rvwvcq] (-3.6619652892562016,0.3528545454545422) circle (2.5pt);

\draw (-8,3.4) node {$\overline{e}$};
\draw (-8,-.1) node {$\overline{a}$};
\draw (-5.3,-.1) node {$\overline{b}$};
\draw (-5.3,3.4) node {$\overline{d}$};
\draw (-10.2,1.6) node {$\overline{f}$};
\draw (-6.3,1.6) node {$\overline{g}$};
\draw (-3.2,1.6) node {$\overline{c}$};

\end{tikzpicture}
\caption{ On the left we have a $2 \times 3$ grid graph, $G$. On the right is the corresponding matching complex $M(G)$. The vertices of $M(G)$ are given by the edges of $G$. There are two maximal 1-simplices (i.e. $(e,b)$ and $(a,d)$) and three shaded maximal 2-simplices.} 
\label{fig:matchingcomp}
\end{figure}
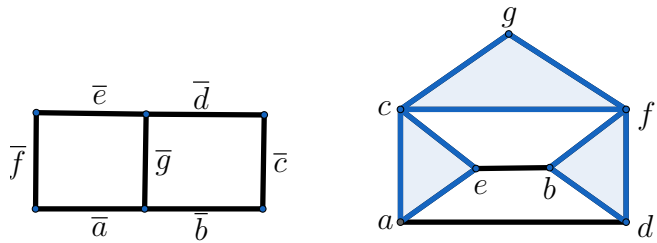


Matching complexes have a robust history and have been studied using both combinatorial and topological methods. 
 In 1992, Bouc \cite{Bouc} introduced the matching complex of the complete graph $K_n$
in relation to the Brown complex and the Quillen complex~\cite{Brown1, Brown2, Quillen}.
Bouc's results were the start of an extensive study of the full matching complex, $M(K_n)$, which has interesting applications (see \cite{Michelle}). 
 
Most of the results about the matching complex $M(K_n)$ have analogues for the 
{\em chessboard complex}, $M(K_{m,n})$. In the last several decades, the Betti numbers, homotopy type, and connectivity of $M(K_n)$ and $M(K_{m,n})$ have been well-studied (see, for instance, the results in \cite{Bouc, Friedman, Ziegler}).
 Studying the connectivity of the chessboard complex \cite{Bjorner_etal, Shareshian_Wachs} was originally motivated by problems in computational geometry \cite{ZV}.

Few other families of graphs have been as well-studied as
$K_n$ and $K_{m, n}$.  In \cite{Kozlov_Trees}, Kozlov determined the homotopy types of the matching complexes of paths and cycles.
In \cite{Braun_Hough}, Braun and Hough used discrete Morse matchings to find bounds on the location of non-trivial homology groups for the $2 \times n$ grid graph. Matsushita built on their results, using purely topological methods to show that the matching complex of a $2 \times n$ grid graph is homotopy equivalent to a wedge of spheres~\cite{Matsushita}.  In 2008, Marietti and Testa proved that the matching complex of a forest is either contractible or homotopy equivalent to a wedge of spheres~\cite[Theorem 4.13]{Marietti_Testa_forests}. 

In \cite{Jakob}, Jonsson proposes studying the topological properties of the matching complexes of honeycomb graphs. A honeycomb graph, also called a hexagonal grid graph, is a subgraph of a unit side length hexagonal tiling of the plane.
An $r \times s \times t$ \emph{honeycomb graph} has $r$ hexagons along its lower left side, $s$ hexagons along its upper left side, and $t$ hexagons along the top (see Figure~\ref{fig:honeycomb}).

Historically, perfect matchings of the honeycomb graph have been studied because of their connections to chemistry (where they are known as Kekul\`e structures \cite{klein_etal, hite_etal}) and to plane partitions \cite{Stanley, Kuperberg}. 
A \emph{perfect matching} of a graph is a matching in which every vertex is incident to exactly one edge of the matching.


\vspace{-1.5cm}
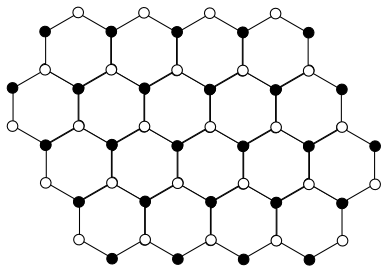
\begin{figure}[ht!]

    \begin{turn}{150}
    \begin{tikzpicture} [ hexa/.style= {shape=regular polygon,
                                   regular polygon sides=6,
                                   minimum size=1cm, draw,
                                   inner sep=0,anchor=south,
                                   fill=white}]

\node[hexa] (hex1) at (0, 0) {};
\foreach \x in {1, 3, 5}
  \fill[color = white, draw = black] (hex1.corner \x) circle[radius=2pt];
  \foreach \x in {2, 4, 6}
  \fill[color = black, draw = black] (hex1.corner \x) circle[radius=2pt];

\node[hexa] (hex2) at (0, {-sin(60)} ) {};
\foreach \x in {1, 3, 5}
  \fill[color = white, draw = black] (hex2.corner \x) circle[radius=2pt];
  \foreach \x in {2, 4, 6}
  \fill[color = black, draw = black] (hex2.corner \x) circle[radius=2pt];

\node[hexa] (hex3) at (0, {-2*sin(60)} ) {};
\foreach \x in {1, 3, 5}
  \fill[color = white, draw = black] (hex3.corner \x) circle[radius=2pt];
  \foreach \x in {2, 4, 6}
  \fill[color = black, draw = black] (hex3.corner \x) circle[radius=2pt];

\node[hexa] (hex4) at (-.76, {-sin(60)/2} ) {};
\foreach \x in {1, 3, 5}
  \fill[color = white, draw = black] (hex4.corner \x) circle[radius=2pt];
  \foreach \x in {2, 4, 6}
  \fill[color = black, draw = black] (hex4.corner \x) circle[radius=2pt];
\node[hexa] (hex5) at (-.76, { -sin(60) -sin(60)/2} ) {};
\foreach \x in {1, 3, 5}
  \fill[color = white, draw = black] (hex5.corner \x) circle[radius=2pt];
  \foreach \x in {2, 4, 6}
  \fill[color = black, draw = black] (hex5.corner \x) circle[radius=2pt];

\node[hexa] (hex6) at (.76, {-sin(60)/2}) {};
\foreach \x in {1, 3, 5}
  \fill[color = white, draw = black] (hex6.corner \x) circle[radius=2pt];
  \foreach \x in {2, 4, 6}
  \fill[color = black, draw = black] (hex6.corner \x) circle[radius=2pt];

\node[hexa] (hex7) at (.76, { -sin(60) -sin(60)/2} ) {};
\foreach \x in {1, 3, 5}
  \fill[color = white, draw = black] (hex7.corner \x) circle[radius=2pt];
  \foreach \x in {2, 4, 6}
  \fill[color = black, draw = black] (hex7.corner \x) circle[radius=2pt];
  
  \node[hexa] (hex8) at (.76, {-sin(60)/2 + sin(60)}) {};
\foreach \x in {1, 3, 5}
  \fill[color = white, draw = black] (hex8.corner \x) circle[radius=2pt];
  \foreach \x in {2, 4, 6}
  \fill[color = black, draw = black] (hex8.corner \x) circle[radius=2pt];

    \node[hexa] (hex9) at (.76, {-sin(60)/2 -2*sin(60)}) {};
\foreach \x in {1, 3, 5}
  \fill[color = white, draw = black] (hex9.corner \x) circle[radius=2pt];
  \foreach \x in {2, 4, 6}
  \fill[color = black, draw = black] (hex9.corner \x) circle[radius=2pt];
  
    \node[hexa] (hex10) at (2*.76, {-sin(60) + sin(60)}) {};
\foreach \x in {1, 3, 5}
  \fill[color = white, draw = black] (hex10.corner \x) circle[radius=2pt];
  \foreach \x in {2, 4, 6}
  \fill[color = black, draw = black] (hex10.corner \x) circle[radius=2pt];
  
    \node[hexa] (hex11) at (2*.76, {-sin(60) + 2*sin(60)}) {};
\foreach \x in {1, 3, 5}
  \fill[color = white, draw = black] (hex11.corner \x) circle[radius=2pt];
  \foreach \x in {2, 4, 6}
  \fill[color = black, draw = black] (hex11.corner \x) circle[radius=2pt];

      \node[hexa] (hex12) at (2*.76, {-sin(60)}) {};
\foreach \x in {1, 3, 5}
  \fill[color = white, draw = black] (hex12.corner \x) circle[radius=2pt];
  \foreach \x in {2, 4, 6}
  \fill[color = black, draw = black] (hex12.corner \x) circle[radius=2pt];
  
      \node[hexa] (hex13) at (2*.76, {-2*sin(60)}) {};
\foreach \x in {1, 3, 5}
  \fill[color = white, draw = black] (hex13.corner \x) circle[radius=2pt];
  \foreach \x in {2, 4, 6}
  \fill[color = black, draw = black] (hex13.corner \x) circle[radius=2pt];

      \node[hexa] (hex14) at (3*.76, {-2*sin(60) + sin(60)/2}) {};
\foreach \x in {1, 3, 5}
  \fill[color = white, draw = black] (hex14.corner \x) circle[radius=2pt];
  \foreach \x in {2, 4, 6}
  \fill[color = black, draw = black] (hex14.corner \x) circle[radius=2pt];
  
      \node[hexa] (hex15) at (3*.76, {-sin(60) + sin(60)/2}) {};
\foreach \x in {1, 3, 5}
  \fill[color = white, draw = black] (hex15.corner \x) circle[radius=2pt];
  \foreach \x in {2, 4, 6}
  \fill[color = black, draw = black] (hex15.corner \x) circle[radius=2pt];

        \node[hexa] (hex16) at (3*.76, {sin(60)/2}) {};
\foreach \x in {1, 3, 5}
  \fill[color = white, draw = black] (hex16.corner \x) circle[radius=2pt];
  \foreach \x in {2, 4, 6}
  \fill[color = black, draw = black] (hex16.corner \x) circle[radius=2pt];

          \node[hexa] (hex18) at (4*.76, { 0}) {};
\foreach \x in {1, 3, 5}
  \fill[color = white, draw = black] (hex18.corner \x) circle[radius=2pt];
  \foreach \x in {2, 4, 6}
  \fill[color = black, draw = black] (hex18.corner \x) circle[radius=2pt];

          \node[hexa] (hex19) at (4*.76, { -sin(60)}) {};
\foreach \x in {1, 3, 5}
  \fill[color = white, draw = black] (hex19.corner \x) circle[radius=2pt];
  \foreach \x in {2, 4, 6}
  \fill[color = black, draw = black] (hex19.corner \x) circle[radius=2pt];

\end{tikzpicture} 
  \end{turn} \vspace{-1cm}

  

\caption{A $3 \times 2 \times 4$ honeycomb graph.}

\label{fig:honeycomb}
\end{figure}

Many methods from different areas of mathematics have been used to study matching complexes (see \cite{Michelle} for a thorough survey). Examples include 
techniques such as long exact sequences \cite{Bouc}, shellability of face posets \cite{Shareshian_Wachs, Woodroofe}, combinatorial topology~\cite{Marietti_Testa_forests}, and, more recently, discrete Morse matchings \cite{Braun_Hough}. For an independence complex, discrete Morse matchings can be found through the Matching Tree Algorithm developed by  Bousquet-Melou, Linusson, and Nevo~\cite{Bou_Lin_Nevo}.

Most results of matching complexes focus on understanding the homology and connectivity (e.g.~\cite{Bjorner_etal, Shareshian_Wachs, Ziegler}).
For a deeper review we refer the interested reader to \cite[Chapter 11]{Jakob}.
Determining the homotopy type of matching complexes has proven to be more difficult. 


\subsection{Our Contributions} 
\label{sec:contributions}
Before examining specific families of graphs, we provide the necessary background
in Section~\ref{sec:background}, including a detailed description of the Matching Tree Algorithm in Section~\ref{mta}. 
Next, we look at
general properties of matching complexes and restrictions that naturally arise on these complexes in Section~\ref{diameter}. 

In Section~\ref{polygon}, we use discrete Morse matchings to find upper and lower bounds on the dimensions of critical cells for the homotopy type of a line of polygons (Theorem \ref{2ngons}) which provides partial results for honeycomb graphs. Further, we give bounds on the dimensions of critical cells for $2 \times 1 \times t$ honeycomb graphs (Theorem \ref{2by1bym}).

In Section~\ref{sec:trees}, we build on a result by Marietti and Testa that shows the matching complex of any forest is either contractible or homotopy equivalent to a wedge of spheres (Theorem \ref{thm:forest}~\cite[Theorem 4.13]{Marietti_Testa_forests}). We do so by using operations from Section~\ref{operations} to determine the explicit homotopy type for general caterpillar graphs (Theorem \ref{thm:gencat} and Theorem \ref{thm:m0}). For a large class of caterpillar graphs, the number of spheres in each dimension has a nice combinatorial interpretation. When the number of spheres does not have a nice formula, we can still determine the explicit homotopy type inductively (Theorem \ref{thm:reallygencat}). We conclude Section~\ref{sec:trees} with a connectivity bound for perfect binary trees, which we prove is tight (Theorem \ref{prop:perfectbinary} and Theorem \ref{prop:ref}).
Section~\ref{sec:future} contains open questions. 

\section{Background}
\label{sec:background}

\subsection{Discrete Morse Theory}

Discrete Morse Theory is a tool that was developed by Forman~\cite{Forman} as a way to find the homotopy type of a complex by pairing its faces. 
These pairings form a sequence of {\em collapses} on the complex, resulting in a homotopy equivalent cell complex.

\begin{definition}
A \emph{partial matching} in a poset $P$ is a partial matching on the underlying graph of the Hasse diagram of $P$. In other words, it is a subset $M \subseteq P \times P$ such that: 
\begin{itemize}
\item $(a,b) \in M$ implies $a \prec b$ and 
\item each $a \in P$ belongs to at most one element in $M$. 
\end{itemize}
When $(a,b) \in M$, we write $a = d(b)$ and $b = u(a)$.\\
A partial matching is \emph{acyclic} if there does not exist a cycle $$a_1 \prec u(a_1) \succ a_2 \prec u(a_2) \succ \cdots \prec u(a_m) \succ a_1 $$ with $m \geq 2$ and all $a_i \in P$ distinct. 

\end{definition}

Given an acyclic partial matching $M$ on a poset $P$, we call an element \emph{critical} if it is unmatched. If every element is matched by $M$, $M$ is called \emph{perfect}. The main theorem of discrete Morse theory~\cite{Kozlov} captures the effect of acyclic matchings.

\begin{theorem}\label{thm:main_mta}
Let $\Delta$ be a polyhedral cell complex and let $M$ be an acyclic matching on the face poset of $\Delta$. Let $c_i$ denote the number of critical $i$-dimensional cells of $\Delta$. The space $\Delta$ is homotopy equivalent to a cell complex $\Delta_c$ with $c_i$ cells of dimension $i$ for each $i \geq 0$, plus a single $0$-dimensional cell in the case where the empty set is paired in the matching. 
\end{theorem}


\subsection{Matching Tree Algorithm (MTA)}\label{mta}

 Let $G$ be a simple graph with vertex set $V(G)$. The Matching Tree Algorithm, due to  Bousquet-Melou, Linusson, and Nevo~\cite{Bou_Lin_Nevo}, is a process which constructs a discrete Morse matching on $\Sigma(\textrm{Ind}(G))$, the face poset of the independence complex of $G$. The \emph{face poset} is defined to be the poset of nonempty faces ordered by inclusion. We shorten  $\Sigma(\textrm{Ind}(G))$ to $\Sigma$ when $G$ is clear. Bousquet-Melou, Linusson, and Nevo motivate their algorithm with the following observation. One way to find a matching of $\Sigma$ is to pick a vertex $v \in V(G)$ with set of neighbors $N(v)$, and then take as the matching all the pairs $(I, I \cup \{ v\})$ for each $I \in \Sigma$ such that $I \cap N(v) = \emptyset$. There may be many unmatched elements, since any element of $\Sigma$ that has nonempty intersection with $N(v)$ will not be in the matching. Choose one of these vertices, and repeat this process as many times as possible. This matching procedure gives rise to a rooted tree, called a matching tree of $\Sigma$, whose nodes keep track of unmatched elements at each step. 

The Matching Tree Algorithm generates a matching tree of $\Sigma$, which is a binary tree whose nodes are either $\emptyset$ or of the form:
 \[
 \Sigma(A,B) = \{I \in \Sigma: A \subseteq I \text{ and } B \cap I = \emptyset\}
 \]
where $A,B \subset V(G)$ with $A \cap B = \emptyset$.
Our version of the Matching Tree Algorithm is modified from its presentation in Braun and Hough~\cite{Braun_Hough} (see Remark \ref{rem:mixitup}). \\

\noindent {\bf Matching Tree Algorithm.} Beginning with the root node $\Sigma(\emptyset,\emptyset)$, at each node $\Sigma(A,B)$ where $A \cup B \not= V(G)$ apply the following procedure:

\begin{enumerate}

\item If there is a vertex $v \in V(G)\setminus (A \cup B)$ such that $N(v) \setminus(A \cup B) = \emptyset$ then $v$ is called a \textit{free vertex}. Give $\Sigma(A,B)$ a single child labeled $\emptyset$.

\item Otherwise, if there is a vertex $v \in V\setminus (A \cup B)$ such that $N(v) \setminus(A \cup B)$ is a single vertex $w$, then $v$ is called a \textit{pivot} and $w$ is called a \textit{matching vertex}. Give $\Sigma(A,B)$ a single child labeled $\Sigma(A \cup \{w\}, B \cup N(w))$.

\item When there is no vertex that satisfies (1) or (2) and $A \cup B \not= V(G)$, we call $\Sigma(A,B)$ \textit{split ready} and the graph induced by $V(G) \setminus (A \cup B)$ a \emph{configuration}. We choose any vertex $v \in V(G) \setminus (A \cup B)$, which we call the \textit{splitting vertex}. Give $\Sigma(A,B)$ two children: $\Sigma(A \cup \{v\},B \cup N(v))$, which we call the \textit{right child}, and $\Sigma(A,B \cup \{v\})$, which we call the \textit{left child}.

\end{enumerate}

\begin{remark}\label{rem:mixitup}
The algorithm as presented in \cite{Braun_Hough} allows one to do steps (1), (2), and (3) in any order, but we found that our ordering worked well for our purposes. Note that when any ordering is allowed, if $w$ is a matching vertex with respect to pivot $v$, we could do either step (2) or step (3). If we choose $w$ as a splitting vertex, $\Sigma(A, B \cup \{w\})$ has $v$ as a free vertex and therefore will have the unique child $\emptyset$. Thus, the only remaining child is $\Sigma(A\cup \{w\}, B  \cup N(w))$, as in step (2). 
\end{remark}

We have the following theorem which is due to Bousquet-Melou, Linusson, and Nevo, but is stated below as it appears in Braun and Hough~\cite{Braun_Hough}.

\begin{theorem} 
\label{thm:MTA}
A matching tree for $G$ yields an acyclic partial matching on the face poset of $\textrm{Ind}(G)$ whose critical cells are given by the non-empty sets $\Sigma(A, B)$ labeling non-root leaves of the matching tree. In particular, for such a set $\Sigma(A, B)$, the set $A$ yields a critical cell in $\textrm{Ind}(G)$.
\end{theorem}

\newpage 

\begin{example}
\label{ex:mta}

Let $G$ be the $1 \times 1 \times 2$ honeycomb graph, whose line graph $L(G)$ is shown below:

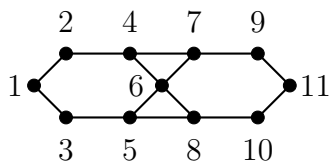
\begin{center}
\begin{tikzpicture}[scale = .85]
\filldraw (0,1) circle (0.1cm); 
\filldraw (1,1) circle (0.1cm); 
\filldraw (-.5,1.5) circle (0.1cm); 
\filldraw (1.5,1.5) circle (0.1cm); 
\filldraw (3.5,1.5) circle (0.1cm); 
\filldraw (2,1) circle (0.1cm);
\filldraw (3,1) circle (0.1cm); 
\filldraw (0,2) circle (0.1cm); 
\filldraw (1,2) circle (0.1cm);
\filldraw (2,2) circle (0.1cm);  
\filldraw (3,2) circle (0.1cm);

\node at  (-.8,1.5) {$1$};
\node at  (1.1,1.5) {$6$};
\node at  (3.9,1.5) {$11$};
\node at  (0,2.5) {$2$};
\node at  (2,2.5) {$7$};
\node at  (3,2.5) {$9$};
\node at  (2,.5) {$8$};
\node at  (3,.5) {$10$};
\node at  (0,.5) {$3$};
\node at  (1,2.5) {$4$};
\node at  (1,.5) {$5$};

\draw[thick] (0, 1) -- (3, 1);
\draw[thick] (0, 2) -- (3, 2);
\draw[thick] (0, 2) -- (-.5,1.5) -- (0, 1);
\draw[thick] (1, 2) -- (1.5,1.5) -- (1, 1);
\draw[thick] (2, 2) -- (1.5,1.5) -- (2, 1);
\draw[thick] (3, 2) -- (3.5,1.5) -- (3, 1);
\end{tikzpicture}
\end{center}

$\Sigma(\emptyset, \emptyset)$ is split ready so we must choose a splitting vertex. Choosing $1$ produces two children, $\Sigma( \emptyset, \{1\})$ and $\Sigma(\{1\}, \{2, 3\} )$, as shown in Figure~\ref{fig:mtaex}.
We first consider the left child. Since $4$ is the only neighbor of $2$ outside of $\{1\}$, we use $4$ as a matching vertex with respect to the pivot $2$. This gives
$\Sigma( \emptyset, \{1\})$ one child that we label 
$\Sigma( \{4 \}, \{1, 2, 6, 7\})$. Similarly, we use $5$ as a matching vertex with respect to the pivot $3$, obtaining the child 
$\Sigma( \{4,5 \}, \{1, 2, 3, 6, 7,8\})$. Finally, we use $11$ as a matching vertex, concluding this branch of the tree with the node
$$\Sigma( \{4,5, 11 \}, \{1, 2, 3 , 6, 7, 8, 9, 10\}).$$

The $\Sigma(\{1\},\{2,3\})$ branch of the tree is split ready. Choosing $6$ as  splitting vertex and continuing the algorithm produces the matching tree shown in Figure
\ref{fig:mtaex}. The resulting matching tree has three leaf nodes:
$$\emptyset,~ \Sigma( \{1,6,11\}; \{ 2,3,4,5,7,8,9,10\}),\text{ and } \Sigma( \{4,5, 11 \}; \{1, 2, 3 , 6, 7, 8, 9, 10\}).$$

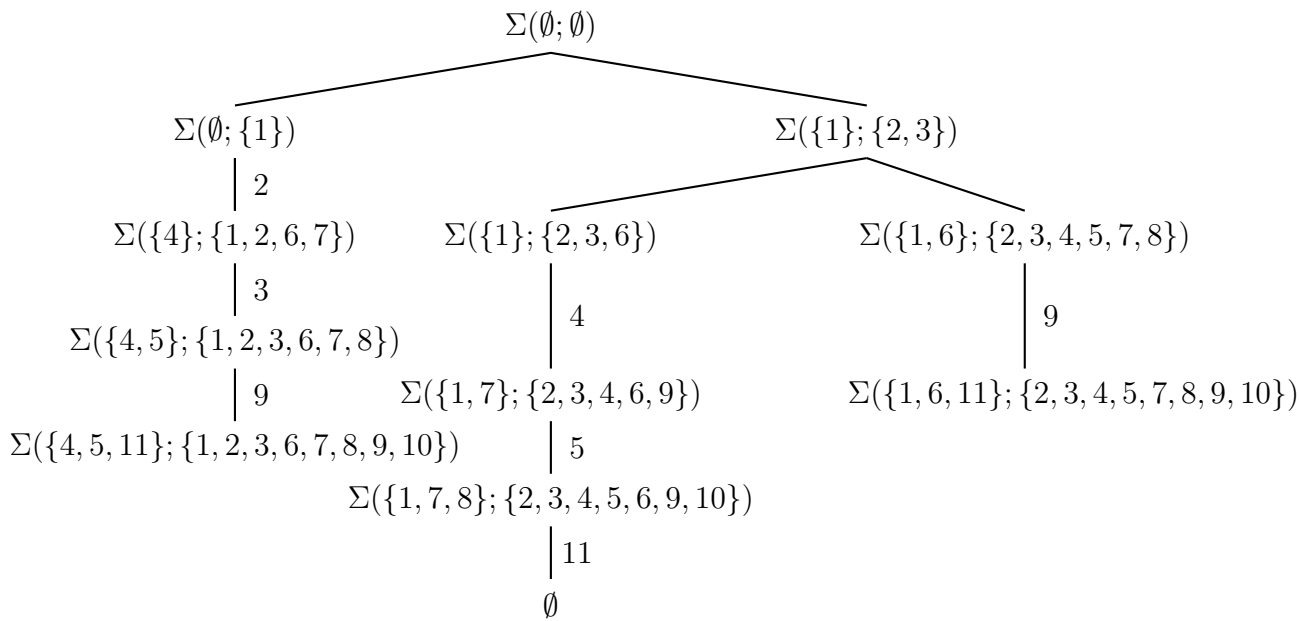
\begin{figure}

\begin{center}
\begin{tikzpicture}[scale = .7, node distance = 3cm]
\node at  (0, 10) {$\Sigma( \emptyset; \emptyset)$};

\draw[thick] (-6,8.5) -- (0, 9.5) -- (6, 8.5);

\node at  (-6, 8) {$\Sigma( \emptyset;  \{1\})$};
\node at  (-6, 6) {$\Sigma( \{4\};  \{1, 2,6, 7\})$};
\node at  (-6, 4) {$\Sigma( \{4,5\}; \{ 1, 2,3, 6,7,8\})$};
\node at  (-6, 2) {$\Sigma( \{4,5,11\};  \{1, 2, 3 ,  6, 7,8, 9, 10\})$};

\node at (-5.5, 7) {$2$};
\node at (-5.5, 5) {$3$};
\node at (-5.5,3) {$9$};

\node at (12.5-3,5-8.5+8) {$9$};

\draw[thick] (-6,7.5) -- (-6, 6.5);
\draw[thick] (-6,7.5) -- (-6, 6.5);
\draw[thick] (-6,5.5) -- (-6, 4.5);
\draw[thick] (-6,3.5) -- (-6, 2.5);

\node at  (6, 8) {$\Sigma(   \{1\}; \{2,3\})$}; 

\node at  (4-4, 6) {$\Sigma(   \{1\};  \{2, 3, 6\})$}; 
\node at  (4-4,3) {$\Sigma(   \{1,7\};  \{2,3,4,6,9\})$}; 
\node at  (4-4, 2-1) {$\Sigma(   \{1,7,8\};  \{2,3,4,5,6,9,10\})$}; 
\node at  (4-4, 0-1) {$\emptyset$}; 

\draw[thick] (4-4,6.5) -- (6, 7.5);
\draw[thick] (4-4,5.5) -- (4-4, 3.5);
\draw[thick] (4-4,2.5) -- (4-4, 2.5-1);
\draw[thick] (4-4,1.5-1) -- (4-4, .5-1);

\node at (4.5-4, 5-0.5) {$4$};
\node at (4.5-4, 3-1) {$5$};
\node at (4.5-4,1-1) {$11$};

\draw[thick] (6,7.5)-- (12-3,6.5);
\draw[thick]  (12-3,5.5) -- (12-3, 3.5);
\node at  (12-3, 6) {$\Sigma( \{1,6\}; \{2,3,4,5,7,8\})$};
\node at  (11.5-3, 3) [text width = 4cm]{$\Sigma(\{1,6, 11 \}; \{2, 3,4,5, 7, 8, 9, 10\})$};

\end{tikzpicture}
\end{center}
\caption{The tree resulting from applying the Matching Tree Algorithm to the graph in Example \ref{ex:mta}. The pivot or free vertex at each step is labeled.  }
\label{fig:mtaex}
\end{figure}

We have two nonempty leaf nodes, which correspond to two critical cells that each have three elements. Since $M(G)$ has two critical cells of dimension 2 and one cell of dimension 0, by Theorem~\ref{thm:main_mta}, $M(G) \simeq S^2 \vee S^2$. 
In this example we were able to determine the homotopy type of $M(G)$ from the Matching Tree Algorithm because the critical cells were all of the same dimension. In general this will not always be the case, but we do always get homological information because if no critical cells appear of a given dimension, then the homology in this dimension must be trivial.

\end{example}

We call the process of applying (1) and (2) from the Matching Tree Algorithm to $\Sigma(A,B)$ until we reach a split ready descendant \textit{split preparing}. In Example~\ref{ex:mta}, when split preparing $\Sigma(\emptyset, \{1\})$, we used $4$, then $5$, and then $11$ as matching vertices. We could have used 5 as a matching vertex before using $4$. We show in the next lemma that the process of split preparing has no effect on the size of the resulting critical cells.

\begin{lemma}\label{split}
Let $G$ be a graph and let $\Sigma (A,B)$ be a leaf of a matching tree of $G$. If $\Sigma (\hat A,\hat B)$ and $\Sigma (\tilde A,\tilde B)$ are two split ready leaves obtained from split preparing $\Sigma (A,B)$, then $\Sigma (\hat A,\hat B) = \emptyset$ if and only if $\Sigma (\tilde A,\tilde B)= \emptyset$. Otherwise, $|\hat A| = |\tilde A|$ and $\hat A\cup\hat B = \tilde A\cup\tilde B$. 
\end{lemma}
\begin{proof}

For a set of vertices $S$, let $\overline{N}(S):=\bigcup\limits_{v \in S} \overline{N}(v)$, where $\overline{N}(v) := v \cup N(v)$. If the lemma does not hold, then we must be able to satisfy one of three conditions:
\begin{enumerate}
\item $\Sigma (\hat A,\hat B) = \emptyset$ but $\Sigma (\tilde A,\tilde B) \not= \emptyset$,
\item $|\hat A| \not= |\tilde A|$, or
\item $\hat A\cup\hat B \not= \tilde A\cup\tilde B$.
\end{enumerate}

We can assume without loss of generality that $A= B = \emptyset$. For one of these conditions to be true, there exists vertex minimal graph $G$ such that $v_1,\ldots,v_k$ and $\tilde v_1,\ldots,\tilde v_h$ are sequences of vertices with the following properties:
\begin{enumerate}[label=(\alph*)]

\item For each $i$, $v_i$ is a matching vertex of $V(G )\setminus \overline{N}(\{v_1,\ldots,v_{i-1}\})$ and $\tilde v_i$ is a matching vertex of $V(G) \setminus \overline{N}(\{\tilde v_1,\ldots,\tilde v_{i-1}\})$.

\item $V(G) \setminus \overline{N}(\{v_1,\ldots,v_{k}\})$ and $V(G) \setminus \overline{N}(\{\tilde v_1,\ldots,\tilde v_{h}\})$ contain no matching vertices.

\item $V(G) \setminus \overline{N}(\{\tilde v_1,\ldots,\tilde v_{h}\})$ contains no free vertices. 

\item If $V(G) \setminus \overline{N}(\{v_1,\ldots,v_{k}\})$ contains no free vertices, then $k \not= h$ 
or \\$V(G) \setminus \overline{N}(\{v_1,\ldots,v_{k}\}) \not= V(G) \setminus \overline{N}(\{\tilde v_1,\ldots,\tilde v_{h}\})$.
\end{enumerate}

To see why these properties must hold, note that $v_1,...,v_k$ and $\tilde v_1,...,\tilde v_h$ are both maximal lists of vertices that can be added to $A$ while split preparing (by properties (a) and (b)). Once there are no more matching vertices, we check for free vertices. If both $V(G) \setminus \overline{N}(\{v_1,\ldots,v_{k}\})$ and $V(G) \setminus \overline{N}(\{\tilde v_1,\ldots,\tilde v_{h}\})$ have a free vertex, we cannot satisfy any of the three conditions, justifying property (c). If only one of them has a free vertex, we satisfy condition (1). If neither has a free vertex, then we need either $k \not= h$ to satisfy condition (2) or $V(G) \setminus \overline{N}(\{v_1,\ldots,v_{k}\}) \not= V(G) \setminus \overline{N}(\{\tilde v_1,\ldots,\tilde v_{h}\})$ to satisfy condition (3). 

We will show by contradiction that these properties cannot all hold. By definition of matching vertices, there must be some pivot $p$ such that $N(p) = \{v_1\}$. If $v_1 = \tilde v_i$ for some $i$, then we can reorder $(\tilde v_1,\ldots,\tilde v_{h})$ by placing $\tilde v_i$ at the beginning without changing $V(G) \setminus \overline{N}(\{\tilde v_1,\ldots,\tilde v_{h}\})$. This means that the sequences $v_2,\ldots,v_k$ and $\tilde v_1,\ldots,\tilde v_{i-1},\tilde v_{i+1},\ldots,\tilde v_h$ would satisfy the above properties on the graph $G \setminus (\overline{N}(v_1))$, contradicting the minimality of $G$.

If $v_1 \not\in \{\tilde v_1,\ldots,\tilde v_{h}\}$ 
then there must be some minimal $i$ such that $p$ is not a pivot in $V(G) \setminus \overline{N}(\{\tilde v_1,\ldots,\tilde v_i\})$. This means at least one of $\{p, v_1\}$ is in $\overline{N}(\{\tilde v_1,\ldots,\tilde v_i\})$ and not in $\overline{N}(\{\tilde v_1,\ldots,\tilde v_{i-1}\})$. This is only possible if (i) $p = \tilde v_i $, (ii) $v_1 = \tilde v_i$, (iii) $p \in N(\tilde v_i)$, or (iv) $v_1 \in N(\tilde v_i)$.


We have already ruled out (ii); since $N(p) = \{v_1\}$, (iii) is the same as (ii); and (iv) would make $p$ a free vertex. Thus, we only need to consider (i), where $p = \tilde v_i $. 

If $p = \tilde v_i$, then $N(v_1) = \{p\}$ is in $V(G) \setminus \overline{N}(\{\tilde v_1,\ldots,\tilde v_{i-1}\})$. Since $p$ is a pivot with respect to $v_1$, $N(p) = \{v_1\}$ 
is also in $V(G) \setminus \overline{N}(\{\tilde v_1,\ldots,\tilde v_{i-1}\})$. This means that $$\overline{N}(\{\tilde v_1,\ldots,\tilde v_{i-1},p,\tilde v_{i+1},\ldots,\tilde v_h\}) = \overline{N}(\{\tilde v_1,\ldots,\tilde v_{i-1},v_1,\tilde v_{i+1},\ldots,\tilde v_h\}).$$ Then, we can reorder $\tilde v_1,\ldots,\tilde v_{i-1},v_1,\tilde v_{i+1},\ldots,\tilde v_h$ so that $v_1$ comes first which leads to the contradiction of the minimality of $G$ as before. 

\end{proof}

\begin{remark}\label{rem:splitorder}
While the process of split preparing has no affect on the size of the resulting critical cells, our choice of splitting vertices affects both the efficiency of the algorithm and the resulting critical cells. The Matching Tree Algorithm
will always produce a discrete Morse matching of the independence complex, but it won't always be the same discrete Morse matching. Two different discrete Morse matchings may give rise to two (homotopy equivalent) cell complexes with a different number or dimensions of cells. For this reason, our proofs in Section~\ref{polygon} involve careful choices of splitting vertices.

\end{remark}

We conclude the section with a lemma that enables us to give a lower bound on the size of the critical cells of a matching complex for a general graph using the Matching Tree Algorithm.

\begin{lemma}\label{lem:buzzerbeater}
Let $G$ be a graph and let $\{v_1,\ldots,v_k\} \subseteq V(G)$ such that $\overline N(v_1)$, $\overline N(v_2)$,$\ldots$, and $\overline N(v_k)$ are pairwise disjoint. There exists a discrete Morse matching of $\textrm{Ind}(G)$ whose critical cells all have at least $k$ elements. 
\end{lemma}

\begin{proof}
Apply the Matching Tree Algorithm without choosing $v_i$ as a splitting vertex for any $i \in [k]$. Assuming the Matching Tree Algorithm has not terminated, there exists
$x \in V(G) \setminus \{v_1, \ldots, v_k \}$ that can be chosen as a splitting vertex because otherwise there is a free vertex $v_i$. Once the algorithm terminates, consider an arbitrary nonempty leaf node $\Sigma(A,B)$. It is possible that no such leaf node exists, in which case the result holds. The critical cell $A$ contains at least one vertex from each $\overline N(v_i)$ for $i \in [k]$. If $v_i \in A$, we are done. Otherwise $v_i \in B$. Since $v_i$ is not a splitting vertex, $v_i$ must be a pivot with respect to some matching vertex $w_i \in N(v_i)$. Therefore, $w_i \in A$ and the result follows.
\end{proof}


\subsection{Graph Operations}\label{operations}
An effective way to study the homotopy type of complexes arising from graphs is to inductively build the graph and study the homotopy type at each step, see for example \cite{Adamaszek, Barmak, Engstrom_ClawFree, Engstrom, Marietti_Testa}. We will use this method in Section~\ref{sec:trees}; here, we present the necessary topological background.
We will state well-known topological results without proof. Further detail can be found in \cite{Hatcher}. 

\begin{remark}
\label{rem:disjoint}
If $G$ and $H$ are disjoint graphs it is clear that
\[M(G\cup H)=M(G)*M(H),\]
where $*$ denotes the \emph{join} operation: $\Delta*\Delta'=\{\sigma\cup\sigma':\ \sigma \in\Delta,\;\sigma'\in\Delta'\}$ for two simplicial complexes $\Delta$ and $\Delta'$.
\end{remark}

For a space $X$, the \emph{suspension} $SX$ is the quotient of $X \times I$ under the identifications $(x_1, 0) \sim (x_2, 0)$ and $(x_1, 1) \sim (x_2, 1)$ for all $x_1, x_2 \in X$. A suspension is homeomorphic to the join of $X$ with two points $\{a,b\}$. The \emph{cone} $CX$ of a space $X$ is $(X \times I)/ (X \times \{0\})$ and is homeomorphic to the join of $X$ with one point.

\begin{remark}\label{rmk:topology} The following properties of suspension and wedge are well-known, see \cite{Bukh,Hatcher, Whitehead}.

\begin{enumerate} 
\item Let $f: X \rightarrow X'$ and $g: Y \rightarrow Y'$ be two homotopy equivalences. Then the join of these maps $f \ast g: X \ast Y \rightarrow X' \ast Y'$ is also a homotopy equivalence.
\item $S^{m} * S^{n} \cong S^{m+n+1}$, and more generally,
$$(S^{a_1} \vee \cdots \vee S^{a_k}) * (S^{b_1} \vee \cdots \vee S^{b_{\ell}}) \simeq \bigvee\limits_{\substack{i = 1, \ldots, k \\ j = 1, \ldots, \ell}} S^{a_i + b_j + 1}$$
\item $S(S^m \vee S^n) \simeq S(S^m) \vee S(S^n) \cong S^{m+1} \vee S^{n+1}$, and more generally, $$S(S^{a_1} \vee \cdots \vee S^{a_k}) \simeq \bigvee\limits_{i=1,\ldots,k} S^{a_i + 1}$$
\end{enumerate}

\end{remark}

\begin{definition}
Let $G$ be a graph and $\overline{x} \in E(G)$. Define the \emph{deletion} of $\overline{x}$ by
\[\text{del}_{G}(\overline{x}) := \{ \overline{e} \in E(G) \mid \overline{e} \cap \overline{x}  = \emptyset \}. \]
 We will abuse notation and refer to the subgraph of $G$ spanned by the deletion of $\overline{x}$ as $\text{del}_{G}(\overline{x})$. 
\end{definition}

\begin{definition}
\label{def:mleg}
We say that a graph $G$ has an \emph{independent $m$-leg}, $m \geq 1$, if there exist edges $\overline{x}_1, \overline{x}_2, \ldots, \overline{x}_m$ and $\overline{y}$ such that $\overline{x}_1, \overline{x}_2, \ldots, \overline{x}_m$ are incident as leaves to the same endpoint of $\overline{y}$ (see Figure~\ref{fig:mleg}).
\end{definition}

Lemma~\ref{lem:m} is a consequence of~\cite[Proposition 3.3]{Marietti_Testa_forests}.

\begin{lemma}
\label{lem:m}
Let $G$ be a graph with an independent $m$-leg, as shown in Figure~\ref{fig:mleg}. If $G_1 = \text{del}_{G}(\overline{x_1})$ and $G_2 = \text{del}_{G}(\overline{y})$, then $M(G) \simeq \left[\bigvee\limits_{m-1}S(M(G_1))\right] \vee S(M(G_2))$.
In particular, when $m = 1$, $M(G) \simeq S(M(G_2))$. 
\end{lemma}

Lemma~\ref{lem:m} and induction can be used to recover the homotopy type for matching complexes of paths~\cite{Kozlov}. Additionally, the next example shows how the lemma can be applied to determine the homotopy type of a small family of trees. 

\begin{example}
\label{ex:onechild}
Consider the family of graphs $\{T_i\}$ where $T_1$ is a path of length 2 and, in general, $T_i$ is constructed from $T_{i-1}$ by adding a path of length 2 to the middle vertex of the most recent path added. Note that $i$ is  the number of paths of length 2 used to construct $T_i$. Figure \ref{fig:onechild} depicts $T_1, T_2,$ and $T_3$.
It is easy to check that $M(T_1) \simeq S^0$, $M(T_2) \simeq S^0$, and $M(T_3) \simeq S^1$. By Lemma~\ref{lem:m}, $M(T_i) \simeq S(M(T_{i-2}))$. For example, 
\begin{itemize}
\item[] $M(T_4) \simeq S(M(T_2)) \simeq S^1$, 
\item[] $M(T_5) \simeq S(M(T_3)) \simeq S^2$, and
\item[] $M(T_6) \simeq S(M(T_4)) \simeq S^2$.
\end{itemize}
In general, for $n \geq 1$, $M(T_n) \simeq S^k$ where $k = \lfloor \frac{n-1}{2} \rfloor$.

\end{example}

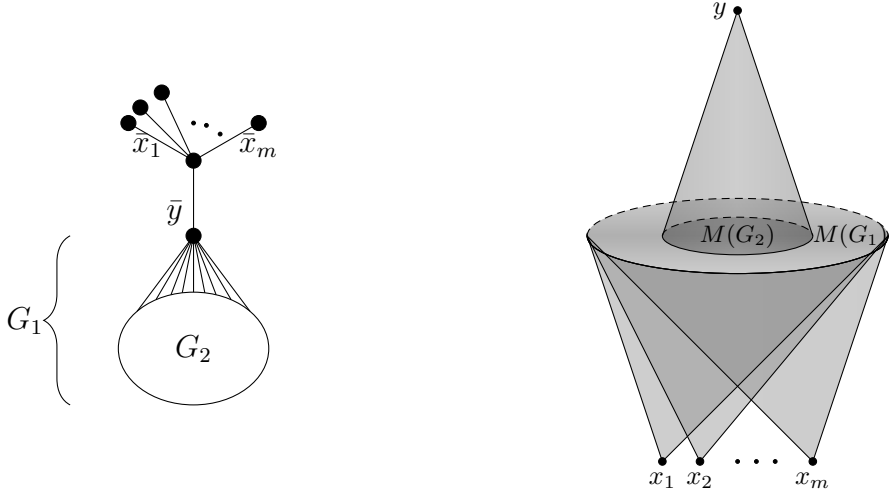
\begin{figure}[htb!] 
\begin{minipage}{.4\textwidth}
\begin{center}
\begin{tikzpicture}[scale = 1]
\filldraw (0,1.5) circle (0.1cm);
\filldraw (0,2.5) circle (0.1cm);

\filldraw (3^.5/2 , 2.5 + .5) circle (0.1cm);
\filldraw (-3^.5/2 , 2.5 + .5) circle (0.1cm);
\filldraw ( -0.4226,0.906 + 2.5) circle (0.1cm);
\filldraw (-2^.5/2 , 2.5 + 2^.5/2) circle (0.1cm);

\filldraw(0.3420/2, 0.939/2 + 2.5) circle (0.025cm);
\filldraw (2^.5/4 , 2.5 + 2^.5/4) circle (0.025cm);
\filldraw (0 , 3) circle (0.025cm);

\draw (0,2.5)--(3^.5/2 , 2.5 + .5) ;
\draw (0,2.5)--(-3^.5/2 , 2.5 + .5) ;
\draw (0,2.5)-- (-2^.5/2 , 2.5 + 2^.5/2) ;
\draw (0,2.5) -- ( -0.4226,0.906 + 2.5) ;
\draw (0, 1.5) -- (0, 2.5);
\draw (0, 1.5) -- (-.75, 0.496);
\draw (0, 1.5) -- (-.5, 0.6495);
\draw (0, 1.5) -- (-.3, 0.7155);
\draw (0, 1.5) -- (-.15, 0.7415);
\draw (0, 1.5) -- (0, 0.75);
\draw (0, 1.5) -- (.15, 0.7415);
\draw (0, 1.5) -- (.3, 0.7155);
\draw (0, 1.5) -- (.5, 0.6495);
\draw (0, 1.5) -- (.75, 0.496);

\node[left] at (0, 1.8) {$\bar{y}$};
\node[below] at (-3^.5/2 + .25 , 2.5 + .5 ) {$\bar{x}_1$};
\node[below] at (3^.5/2 , 2.5 + .5 ) {$\bar{x}_m$};

\node at (0, 0) {$G_2$};

\draw (0, 0) ellipse (1 and 0.75);

\draw [decorate,decoration={brace,amplitude=10pt},xshift=-4pt,yshift=0pt]
(-1.5,-0.75) -- (-1.5,1.5) node [black,midway,xshift=-0.6cm] 
{ $G_1$};
\end{tikzpicture}
\end{center}
\end{minipage}
\begin{minipage}{.55\textwidth}
\begin{center}
\begin{tikzpicture}
\fill[top color=gray!50!black,bottom color=gray!10,middle color=gray,shading=axis,opacity=0.25] (0,0) circle (1cm and 0.25cm); 
\fill[top color=gray!10,bottom color=gray!10,middle color=gray,shading=axis,opacity=0.25] (0,0) circle (2cm and 0.5cm); 

\fill[left color=gray!50,right color=gray!50,middle color=gray!5,shading=axis,opacity=0.25] (1,0) -- (0,3) -- (-1,0) arc (180:360:1cm and 0.25cm); 

\fill[left color=gray!5,right color=gray!10,middle color=gray!5,shading=axis,opacity=0.25] (2,0) -- (-1,-3) -- (-2,0) arc (180:360:2cm and 0.5cm); 
\fill[left color=gray!5,right color=gray!10,middle color=gray!5,shading=axis,opacity=0.25] (2,0) -- (-0.5,-3) -- (-2,0) arc (180:360:2cm and 0.5cm); 
\fill[left color=gray!5,right color=gray!10,middle color=gray!5,shading=axis,opacity=0.25] (2,0) -- (1,-3) -- (-2,0) arc (180:360:2cm and 0.5cm); 

\draw (-2,0) arc (180:360:2cm and 0.5cm) -- (-.5,-3) -- cycle; 
\draw (-2,0) arc (180:360:2cm and 0.5cm) -- (-1,-3) -- cycle; 
\draw (-2,0) arc (180:360:2cm and 0.5cm) -- (1,-3) -- cycle; 
\draw (-1,0) arc (180:360:1cm and 0.25cm) -- (0,3) -- cycle; 
\draw[densely dashed] (-2,0) arc (180:0:2cm and 0.5cm); 
\draw[densely dashed] (-1,0) arc (180:1:1cm and 0.25cm); 

\filldraw (0,3) circle (0.05cm);
\filldraw (1,-3) circle (0.05cm);
\filldraw (-1,-3) circle (0.05cm);
\filldraw (-.5,-3) circle (0.05cm);

\filldraw (0.5,-3) circle (0.025cm);
\filldraw (0.25,-3) circle (0.025cm);
\filldraw (0,-3) circle (0.025cm);

\node[left] at (0, 3) {\footnotesize{$y$}};
\node[below] at (-1, -3) {\footnotesize{$x_1$}};
\node[below] at (-.5, -3) {\footnotesize{$x_2$}};
\node[below] at (1, -3) {\footnotesize{$x_m$}};

\node at (0, 0) {\scriptsize{$M(G_2)$}};
\node at (1.5, 0) {\scriptsize{$M(G_1)$}};

\end{tikzpicture}

\end{center}
\end{minipage}
\caption{A graph $G$ with edges $\bar{x}_1, \ldots, \bar{x}_m$ and $\bar{y}$ that form an independent $m$-leg (shown left) and its matching complex $M(G)$ (shown right).}
\label{fig:mleg}
\end{figure}


\begin{figure}[h!]

\begin{center}
\begin{tikzpicture}[scale = .4]

\filldraw (2, 0) circle (0.1cm);
\filldraw (0, 0) circle (0.1cm);
\filldraw (0, 2) circle (0.1cm);

\draw (2, 0) -- (0, 0) -- (0, 2);

\node[below] at (1, -1) {$T_1$};

\filldraw (9, 0) circle (0.1cm);
\filldraw (7, 0) circle (0.1cm);
\filldraw (5, 0) circle (0.1cm);
\filldraw (7, 2) circle (0.1cm);
\filldraw (5, 2) circle (0.1cm);

\draw (9, 0) -- (7, 0) -- (5, 0) -- (5, 2);
\draw (7, 0) -- (7, 2);

\node[below] at (7, -1) {$T_2$};

\filldraw (18, 0) circle (0.1cm);
\filldraw (16, 0) circle (0.1cm);
\filldraw (16, 2) circle (0.1cm);
\filldraw (14, 0) circle (0.1cm);
\filldraw (12, 0) circle (0.1cm);
\filldraw (14, 2) circle (0.1cm);
\filldraw (12, 2) circle (0.1cm);

\draw (18, 0) -- (16, 0) -- (14, 0) -- (12, 0) -- (12, 2);
\draw (14, 0) -- (14, 2);
\draw (16, 0) -- (16, 2);

\node[below] at (15, -1) {$T_3$};

\end{tikzpicture}

\end{center}

\caption{The family of graphs described in Example \ref{ex:onechild}.}
\label{fig:onechild}
\end{figure}
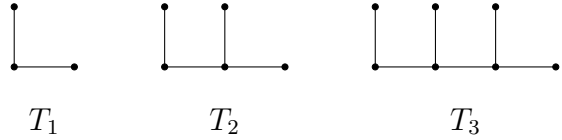


\section{Connectedness and Diameter} \label{diameter}
There are inherent restrictions on the types of matching complexes that arise when considering the structure of graphs. For example, Jonsson~\cite{Jakob} gave lower bounds on the depth of $M(G)$ under certain conditions, such as in the case where $G$ admits a perfect matching. More recently, Bayer, Goeckner, and Jeli\'c Milutinovi\'c  completely classified matching complexes that are combinatorial manifolds with and without boundaries~\cite{bestpaper}.  Proposition~\ref{prop:conn} characterizes when the matching complex is disconnected. 

\begin{prop}
\label{prop:conn}
Let $M(G)$ denote the matching complex of a graph $G$. 
$M(G)$ is disconnected if and only if there exists a nonempty set $S \subset E(G)$ such that $S \not= E(G)$, and every edge of $E(G) \setminus S$ is incident to every edge of $S$. 

\end{prop}

\begin{proof}

$(\Leftarrow)$ Assume there is a set $\emptyset \subsetneq S \subsetneq E(G)$ such that every edge of $E(G) \setminus S$ is incident to every edge of $S$. Then there cannot be a matching containing edges of $S$ and edges of $E(G) \setminus S$ and thus the matching complex is disconnected.

($\Rightarrow$) Assume that there does not exist a set $\emptyset \subsetneq S \subsetneq E(G)$ such that every edge of $E(G) \setminus S$ is incident to every edge of $S$. Let $\overline{e}_1, \overline{e}_5 \in E(G)$. We need to show that there is a path between $e_1$ and $e_5$ in the matching complex. 

By assumption, there is an edge $\overline e_2$ not incident to $\overline e_1$ and an edge $\overline e_4$ not incident to $\overline e_5$. In order to prevent a path from $e_1$ to $e_5$  in the matching complex, these four edges must be distinct and form a cycle. 
Furthermore, by assumption, $\overline e_1$ and $\overline e_2$ cannot both be incident to every edge of $G$. The same statement holds for $\overline e_4$ and $\overline e_5$. 
Thus, there must be an edge $\overline e_3$ that is incident to at most one of the 4 vertices formed by the cycle $\overline{e}_1,\overline{e}_5,\overline{e}_2,\overline{e}_4$.
Regardless of which vertex $\overline{e_3}$ is incident to, there is a path from $e_1$ to $e_5$ in the matching complex.

\end{proof}

The \emph{diameter} of a graph is the maximum over the minimum distance between any two vertices. The diameter of a complex is the diameter of its 1-skeleton. 
The following corollary is immediate from the proof of Proposition~\ref{prop:conn}.

\begin{corollary} There is a path of length 
less than or equal to 4 between any two vertices of the 1-skeleton of any connected matching complex. That is, any connected matching complex has diameter less than or equal to 4.
\end{corollary}

This bound is achievable.  Figure \ref{diam4} shows a graph whose matching complex is a path on five vertices.

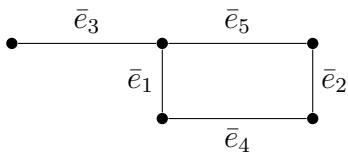
\begin{figure}
\begin{center}
\begin{tikzpicture}
    
    \tikzstyle{every node} = [circle,fill,inner sep=1pt,minimum size = 1.5mm]
    \node (a) at (4,5) {};
    \node(b) at (6,5){} ;
    \node(c) at (6,6){};
    \node(d) at (4,6){};
    \node(e) at (2,6){};

    \draw (a)--(b) node [midway, below, fill = none] {$\bar{e}_4$};
    \draw (b)--(c) node [midway, right, fill = none] {$\bar{e}_2$};
    \draw (c)--(d) node [midway, above, fill = none] {$\bar{e}_5$};
    \draw (d)--(a) node [midway, left, fill = none] {$\bar{e}_1$};
    \draw (d)--(e) node [midway, above, fill = none] {$\bar{e}_3$};
    
\end{tikzpicture}
\caption{The matching complex of the graph shown above is $P_5$, which has diameter 4.} 
\label{diam4}
\end{center}
\end{figure}

Furthermore, the matching complexes of a large class of graphs have  diameter 2.


\begin{prop} If $G$ has at least two incident edges, but no pair $\ov{e},\ov{e}'$ such that every edge in $E(G)\setminus\{\ov{e},\ov{e}'\}$ is incident to either $\ov{e}$ or $\ov{e}'$, then the matching complex of $G$ has diameter 2.
\end{prop}

\begin{proof}
For two incident edges $\ov{h_1}$ and $\ov{h_2}$ in $G$ the distance between the two corresponding vertices $h_1$ and $h_2$ on the 1-skeleton of $M(G)$ must be at least 2. Therefore, the diameter is at least 2. Furthermore, for any pair of edges $\ov{e_1}$ and $\ov{e_2}$ in $G$ there exists a third edge $\ov{g}$ that is not incident to either $\ov{e_1}$ or $\ov{e_2}$, so $\{e_1,g, e_2\}$ is a path in the 1-skeleton of $M(G)$ from $e_1$ to $e_2$. 
Hence, $M(G)$ has diameter 2.  

\end{proof}

\section{Polygon Tilings} \label{polygon}

\begin{figure}[!htb] 
\begin{center}
\begin{tikzpicture}[scale = .85]
\filldraw (0.5,0) circle (0.1cm); 
\filldraw (0,0.5) circle (0.1cm); 
\filldraw (0,1.5) circle (0.1cm); 
\filldraw (0.5,2) circle (0.1cm); 

\filldraw (1.5,2) circle (0.1cm); 
\filldraw (2,1.5) circle (0.1cm);
\filldraw (2,0.5) circle (0.1cm); 
\filldraw (1.5,0) circle (0.1cm); 
\filldraw (2.5,2) circle (0.1cm); 
\filldraw (2.5,0) circle (0.1cm);

\filldraw (3.5,2) circle (0.1cm); 
\filldraw (4,1.5) circle (0.1cm);
\filldraw (4,0.5) circle (0.1cm); 
\filldraw (3.5,0) circle (0.1cm); 
\filldraw (4.5,2) circle (0.1cm); 
\filldraw (4.5,0) circle (0.1cm);

\filldraw (5.75,1) circle (0.025cm); 
\filldraw (6,1) circle (0.025cm); 
\filldraw (6.25,1) circle (0.025cm);

\filldraw (8,1.5) circle (0.1cm);
\filldraw (8,0.5) circle (0.1cm); 
\filldraw (7.5,2) circle (0.1cm); 
\filldraw (7.5,0) circle (0.1cm);

\draw[thick] (0.5, 0) -- (0,0.5)-- (0,1.5)-- (0.5,2);
\draw[thick] [dashed](0.5,2) -- (1.5,2);
\draw[thick] [dashed](0.5,0) -- (1.5,0);
\draw[thick] (1.5, 2) -- (2,1.5) -- (2,0.5) -- (1.5,0);
\draw[thick] (2.5, 0) -- (2,0.5)-- (2,1.5)-- (2.5,2);
\draw[thick] [dashed](2.5,2) -- (3.5,2);
\draw[thick] [dashed](2.5,0) -- (3.5,0);
\draw[thick] (3.5, 2) -- (4,1.5) -- (4,0.5) -- (3.5,0);
\draw[thick] (4.5, 0) -- (4,0.5)-- (4,1.5)-- (4.5,2);
\draw[thick] [dashed](4.5,2) -- (5.5,2);
\draw[thick] [dashed](4.5,0) -- (5.5,0);
\draw[thick] [dashed](6.5,2) -- (7.5,2);
\draw[thick] [dashed](6.5,0) -- (7.5,0);
\draw[thick] (7.5, 2) -- (8,1.5) -- (8,0.5) -- (7.5,0);

\end{tikzpicture}
\caption{A graph of $t$ $2n$-gons.}
\label{original}
\end{center}
\end{figure}
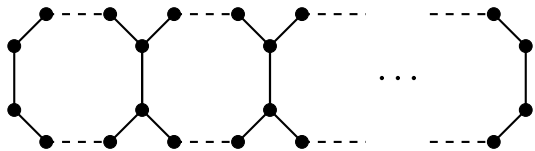

In~\cite{Jakob}, Jonsson proposes studying the topology of matching complexes of honeycomb graphs. In this section we will consider polygonal tilings and $2 \times 1 \times t$ honeycombs for $t \geq 1$, providing partial results for the matching complexes of honeycomb graphs. 

Consider a graph of $t$ $2n$-gons for $n \ge 2$ and $t \geq 1$ arranged in a line as shown in Figure \ref{original}. 
Jonsson explored the matching complex of this graph when $n=2$, so the graph is a line of quadrilaterals \cite{JonssonMatchingGrids}, showing that the matching complex is
at least $d^{\min}_{t}$-connected, where 
\[d^{\min}_{t} = 
\begin{cases}
2\left\lfloor \dfrac{t-2}{3}\right\rfloor + 2  & \text{ if } t \equiv 0~(\bmod ~3)\\
\left\lfloor \dfrac{2t}{3} \right\rfloor & \text{ otherwise. } 
\end{cases}
\]
A topological space is {\em $k$-connected} if the higher homotopy groups $\pi_\ell$ vanish for all dimensions $\ell \leq k$.
Braun and Hough \cite{Braun_Hough} gave a proof of this result using the Matching Tree Algorithm. They also showed that the matching complex of $t$ 4-gons has no $d$-dimensional cells for $d > d_{t}^{\max}$, where
\[d^{\max}_{t} = \left\lfloor \dfrac{3t-1}{4} \right\rfloor. 
\]
Matsushita \cite{Matsushita} concluded the study of the matching complex of $t$ 4-gons by determining the homotopy type.

The following theorem extends these results by considering when $n > 2$. Note that the theorem only considers when $t > 1$, since the homotopy type of matching complexes of cycle graphs are known.

\begin {theorem} 
\label{2ngons}

Let $t>1$ and $n > 2$. 
Define

\[d^{\min} = 
\begin{cases}
\dfrac{2nt}{3}-t  & \text{ if } n \equiv 0 ~(\bmod ~3) \\
\dfrac {2nt+t}3-t & \text{ if } n \equiv 1 ~(\bmod ~3) \\
\dfrac{2nt-t}{3}- \left\lfloor \dfrac {t+1}2 \right\rfloor & \text{ if } n \equiv 2 ~(\bmod ~3)
\end{cases}
\]
and
\[d^{\max} = 
\begin{cases}
\dfrac{2nt}{3}-\left\lfloor \dfrac{t}{2} \right\rfloor -1  & \text{ if } n \equiv 0 ~(\bmod ~3) \\
\dfrac {2nt+t}3-t  & \text{ if } n \equiv 1 ~(\bmod ~3) \\
\dfrac{2nt-t}{3} - 1 & \text{ if } n \equiv 2 ~(\bmod ~3)
\end{cases}.
\]

The matching complex of a graph of $t$ $2n$-gons is homotopy equivalent to a space with no $d$-dimensional cells, where $0 < d < d^{\min}$ or $d > d^{\max}$. Consequently, the connectivity is at least $d^{\min} -1$.
\end{theorem}

\begin{remark}
 In the case where $n \equiv 1 ~(\bmod ~3)$, Theorem~\ref{2ngons} determines the homotopy type of the matching complex of $t$ $2n$-gons (see 
Corollary~\ref{cor:nequiv1}).
After our work was first posted on the arXiv, Matsushita \cite{MatPolygonal}
 provided an alternative proof of this result and additionally determined the homotopy type of the matching complex of $t$ $2n$-gons recursively when $n \equiv 0$ and $n \equiv 2 ~(\bmod ~3)$ . His results show that the quantity $d^{\min}$ in Theorem~\ref{2ngons} is tight when $n \equiv 0$ but not when  $n \equiv 2$ (see Theorems 3.8 and 3.16 of \cite{MatPolygonal}). Matsushita's methods are different from ours. While we rely on the Matching Tree Algorithm, he determines the homotopy type inductively by cleverly deleting vertices from the line graph. A benefit of our proof using the Matching Tree Algorithm is its preparation for the similar but more technical proof of Theorem \ref{2by1bym}. 
\end{remark}

\begin{proof}

We will use the Matching Tree Algorithm to establish both bounds. Since this algorithm finds matchings on the face poset of the independence complex of a graph, 
we apply the algorithm to the line graph, which we label as shown in Figure \ref{fig:case1}, when $j=1$.

Recall by Lemma~\ref{split} and Remark~\ref{rem:splitorder} that the size of the critical cells of our matching tree depends only on the order that we choose splitting vertices and not how we split prepare. We apply the Matching Tree Algorithm using the following procedure.

For each split ready leaf node $\Sigma(A, B)$:

\begin{itemize}
\item[] Step 1: Choose the smallest $a_j$ not yet assigned to $A$ or $B$ as our splitting vertex. This produces two leaves, $\Sigma(A, B \cup \{a_j\})$ and $\Sigma(A \cup \{a_j\}, B \cup N(a_j))$.
\item[] Step 2: Split prepare each leaf. 
\end{itemize}

At Step 1 of this algorithm, there will be two possible cases for $V(G) \setminus (A \cup B)$, as seen in Figure~\ref{fig:case1} and Figure~\ref{fig:case2}. 

We split prepare based on
(i) the case, (ii) whether the leaf is the left or right child of $\Sigma(A, B)$, and (iii) the value of $n$ (mod 3). 

The follow sets of $V(G)$ make it easier to describe which vertices are added to $A$ when split preparing and will aid in the analysis of the critical cells.

\begin{itemize}
\item[]$E_j = \{ a_j \text{ as well as } b_{(j,k)} \text{ and } c_{(j,k)}\text{ for all } k \equiv 0 \pmod{3} \}$
\item[]$F_j = \{ b_{(j,k)} \text{ and } c_{(j,k)}\text{ for all } k \equiv 1 \pmod{3} \}$
\item[]$G_j = \{ b_{(j,k)} \text{ and } c_{(j,k)}\text{ for all } k \equiv 2 \pmod{3} \}$
\item[]$H_j = \{ b_{(j,k)} \text{ and } c_{(j,k)}\text{ for all } k \equiv 0 \pmod{3} \}$
\end{itemize}

We begin with Case 1 (see Figure~\ref{fig:case1}).\\

\noindent {\bf Case 1} (deg($a_j$) = 2):

\begin{figure}[!htb] 
\begin{center}
\begin{tikzpicture}[scale = .85]
\filldraw (1,0) circle (0.1cm); 
\filldraw (0,1) circle (0.1cm); 
\filldraw (1,2) circle (0.1cm); 
\filldraw (2.4,0) circle (0.1cm); 
\filldraw (2.4,2) circle (0.1cm); 
\filldraw (3.4,1) circle (0.1cm);
\filldraw (4.4,2) circle (0.1cm); 
\filldraw (4.4,0) circle (0.1cm); 
\filldraw (5.8,2) circle (0.1cm); 
\filldraw (5.8,0) circle (0.1cm);
\filldraw (6.8,1) circle (0.1cm);  
\filldraw (7.8,2) circle (0.1cm);
\filldraw (7.8,0) circle (0.1cm);

\filldraw (9.2,1) circle (0.025cm); 
\filldraw (9.4,1) circle (0.025cm); 
\filldraw (9.6,1) circle (0.025cm); 

\filldraw (11,2) circle (0.1cm); 
\filldraw (11,0) circle (0.1cm);
\filldraw (12,1) circle (0.1cm);  
\filldraw (13,2) circle (0.1cm);
\filldraw (13,0) circle (0.1cm); 
\filldraw (15.4,1) circle (0.1cm);  
\filldraw (14.4,2) circle (0.1cm);
\filldraw (14.4,0) circle (0.1cm); 

\node at  (-.4,1) {$a_j$};
\node at  (2.8,1) {$a_{j+1}$};
\node at  (1,2.5) {$b_{(j,1)}$};
\node at  (2.4,2.5) {$b_{(j,n-1)}$};
\node at  (4.4,2.5) {$b_{(j+1,1)}$};
\node at  (1,-.5) {$c_{(j,1)}$};
\node at  (2.4,-.5) {$c_{(j,n-1)}$};
\node at  (4.4,-.5) {$c_{(j+1,1)}$};

\node at  (11.55, 1) {$a_{t}$};
\node at  (14.6, 1) {$a_{t+1}$};
\node at  (13,2.5) {$b_{(t,1)}$};
\node at  (15,2.5) {$b_{(t,n-1)}$};
\node at  (13,-.5) {$c_{(t,1)}$};
\node at  (15,-.5) {$c_{(t,n-1)}$};

\draw[thick] (1, 0) --(0, 1)--(1,2);
\draw[thick][dashed] (1,0) -- (2.4, 0);
\draw[thick][dashed] (1,2) -- (2.4, 2);
\draw[thick] (2.4, 0) -- (3.4,1)-- (2.4,2);
\draw[thick] (2.4,0) -- (4.4, 0) -- (3.4,1)-- (4.4,2) -- (2.4,2);
\draw[thick][dashed] (4.4,0) -- (5.8, 0);
\draw[thick][dashed] (4.4,2) -- (5.8, 2);
\draw[thick] (5.8, 0) -- (6.8,1)-- (5.8,2);
\draw[thick] (5.8,0) -- (7.8, 0) -- (6.8,1)-- (7.8,2) -- (5.8,2);
\draw[thick][dashed] (7.8,0) -- (8.8, 0);
\draw[thick][dashed] (7.8,2) -- (8.8, 2);

\draw[thick][dashed] (10,0) -- (11, 0);
\draw[thick][dashed] (10,2) -- (11, 2);
\draw[thick] (11, 0) -- (12,1)-- (11,2);
\draw[thick] (11,0) -- (13, 0) -- (12,1)-- (13,2) -- (11,2);
\draw[thick][dashed] (13,0) -- (14.4, 0);
\draw[thick][dashed] (13,2) -- (14.4, 2);
\draw[thick] (14.4, 0) -- (15.4,1)-- (14.4,2);

\end{tikzpicture}
\caption{$G \setminus (A \cup B)$ in Case 1}
\label{fig:case1}
\end{center}
\end{figure}
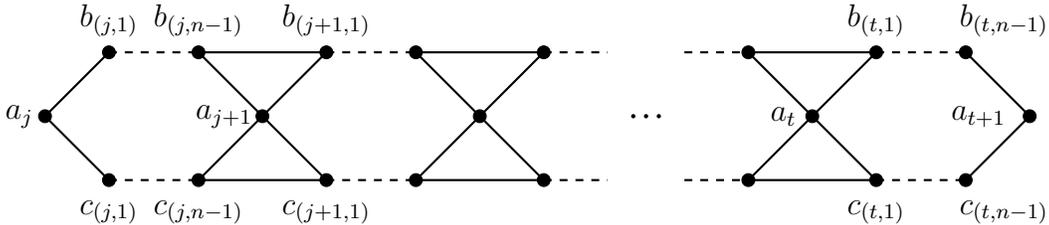

\begin{itemize}
\item[] \underline{Left child}: The left child of $\Sigma(A, B)$ is $\Sigma(A, B \cup \{a_j \})$. Since  $b_{(j,1)}$ and $c_{(j,1)}$ are pivots we add $b_{(j,2)}$ and $c_{(j,2)}$ to $A$ and the neighbors of $b_{(j,2)}$ and $c_{(j,2)}$ to $B$. Continuing to split prepare, we add the remaining vertices of $G_j$ to $A$ and $N(G_j)$ to $B$. 

\begin{itemize}
\item If $n \equiv 2 \pmod{3}$, the largest $k$ for which $\{b_{(j,k)},c_{(j,k)}\} \subseteq G_j$ is $n-3$ and
the last vertices added to $B$ are $b_{(j,n-2)},c_{(j,n-2)}$.
When $j = t$, $b_{(j,n-1)}$ and $c_{(j,n-1)}$ are pivots with $a_{t+1}$ as the associated matching vertex, so we add $a_{t+1}$ to $A$ and produce a critical cell. When $j \neq t$, we go to Case 2 after increasing $j$ by one. 

\item If $n \equiv 1 \pmod{3}$, the largest $k$ for which $\{b_{(j,k)},c_{(j,k)}\} \subseteq G_j$ is $n-2$ and
the last vertices added to $B$ are $b_{(j,n-1)},c_{(j,n-1)}$.
When $j = t$, $a_{t+1}$ has no neighbors in $G \setminus(A \cup B)$, so we do not get a critical cell. When $j \neq t$, we go to Case 1 after increasing $j$ by one.

\item If $n \equiv 0 \pmod{3}$, the largest $k$ for which $\{b_{(j,k)},c_{(j,k)}\} \subseteq G_j$ is $n-1$. When $j=t$, $a_{t+1}$ is added to $B$ and we have a critical cell. When $j \neq t$, unlike in the other cases, there are more pivots before we reach a split ready leaf. In this case, $a_{j+1}$, $b_{(j+1,1)}$, and $c_{(j+1,1)}$ are all added to $B$, and the situation is analogous to the right child when $a_{j+1}$ is used as a splitting vertex except $a_{j+1}$ is in $B$ rather than $A$. In particular, instead of adding $E_{j+1}$ to $A$, we add $H_{j+1}$ and then go to Case 2 after increasing $j$ by one.

\end{itemize}

\item[]\underline{Right child}: The right child of $\Sigma(A, B)$ is $\Sigma(A \cup \{ a _j\}, B \cup N(a_j))$. We have a situation similar to the left child, but shifted by one. We add $E_{j}$ to $A$ and $N(E_j)$ to $B$. 
\begin{itemize}
\item If $n \equiv 2 \pmod{3}$, the largest $k$ for which $\{b_{(j,k)},c_{(j,k)}\} \subseteq E_j$ is $n-2$. As in the $n \equiv 1 \pmod{3}$ case for the left child, when $j = t$, $a_{t+1}$ has no neighbors in $G \setminus(A \cup B)$, so we do not get a critical cell. When $j \neq t$, we go to Case 1 after increasing $j$ by one.
\item If $n \equiv 1 \pmod{3}$, the largest $k$ for which $\{b_{(j,k)},c_{(j,k)}\} \subseteq E_j$ is $n-1$. As in the $n \equiv 0 \pmod{3}$ case for the left child, when $j=t$, $a_{t+1}$ is added to $B$ and we have a critical cell. When $j \neq t$, $a_{j+1}, b_{(j+1, 1)}$ and $c_{(j+1, 1)}$ are all added to $B$, so there are more pivots before we reach a split ready leaf. Continuing to split prepare, we add $H_{j+1}$ to $A$ and $N(H_{j+1})$ to $B$. Then we add $H_{j+2}$ to $A$ and $N(H_{j+2})$ to $B$. This continues until we add $H_{t}$ to $A$ and $N(H_t)$ to $B$.
\item If $n \equiv 0 \pmod{3}$, the largest $k$ for which $\{b_{(j,k)},c_{(j,k)}\} \subseteq E_j$ is $n-3$ (unless $n=3$, in which case $E_j = \{a_j\}$). As in the $n \equiv 2 \pmod{3}$ case for the left child, when $j = t$, $b_{(j,n-1)}$ and $c_{(j,n-1)}$ are pivots with $a_{t+1}$ as the associated matching vertex, so we add $a_{t+1}$ to $A$ and produce a critical cell. When $j \neq t$, we go to Case 2 after increasing $j$ by one. \\

\end{itemize}
\end{itemize}

\noindent {\bf Case 2} (deg($a_j$) = 4):

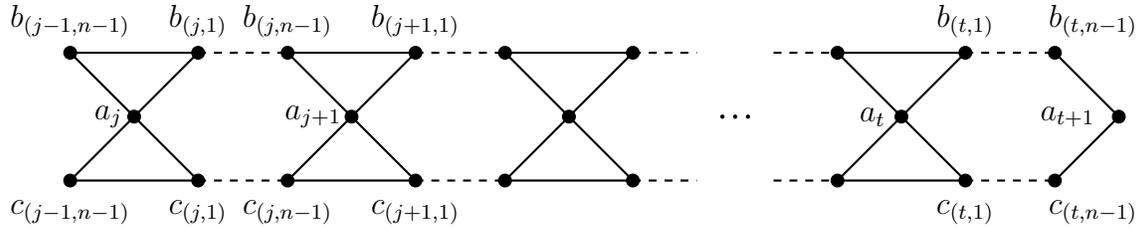
\begin{figure}[!htb] 
\begin{center}
\begin{tikzpicture}[scale = .85]
\filldraw (-1,0) circle (0.1cm); 
\filldraw (-1,2) circle (0.1cm); 
\filldraw (1,0) circle (0.1cm); 
\filldraw (0,1) circle (0.1cm); 
\filldraw (1,2) circle (0.1cm); 
\filldraw (2.4,0) circle (0.1cm); 
\filldraw (2.4,2) circle (0.1cm); 
\filldraw (3.4,1) circle (0.1cm);
\filldraw (4.4,2) circle (0.1cm); 
\filldraw (4.4,0) circle (0.1cm); 
\filldraw (5.8,2) circle (0.1cm); 
\filldraw (5.8,0) circle (0.1cm);
\filldraw (6.8,1) circle (0.1cm);  
\filldraw (7.8,2) circle (0.1cm);
\filldraw (7.8,0) circle (0.1cm);

\filldraw (9.2,1) circle (0.025cm); 
\filldraw (9.4,1) circle (0.025cm); 
\filldraw (9.6,1) circle (0.025cm); 

\filldraw (11,2) circle (0.1cm); 
\filldraw (11,0) circle (0.1cm);
\filldraw (12,1) circle (0.1cm);  
\filldraw (13,2) circle (0.1cm);
\filldraw (13,0) circle (0.1cm); 
\filldraw (15.4,1) circle (0.1cm);  
\filldraw (14.4,2) circle (0.1cm);
\filldraw (14.4,0) circle (0.1cm); 

\node at  (-1,2.5) {$b_{(j-1,n-1)}$};
\node at  (-1,-.5) {$c_{(j-1,n-1)}$};
\node at  (-.4,1) {$a_j$};
\node at  (2.8,1) {$a_{j+1}$};
\node at  (1,2.5) {$b_{(j,1)}$};
\node at  (2.4,2.5) {$b_{(j,n-1)}$};
\node at  (4.4,2.5) {$b_{(j+1,1)}$};
\node at  (1,-.5) {$c_{(j,1)}$};
\node at  (2.4,-.5) {$c_{(j,n-1)}$};
\node at  (4.4,-.5) {$c_{(j+1,1)}$};

\node at  (11.55, 1) {$a_{t}$};
\node at  (14.6, 1) {$a_{t+1}$};
\node at  (13,2.5) {$b_{(t,1)}$};
\node at  (15,2.5) {$b_{(t,n-1)}$};
\node at  (13,-.5) {$c_{(t,1)}$};
\node at  (15,-.5) {$c_{(t,n-1)}$};

\draw[thick] (1, 0) --(-1,0) -- (0, 1) -- (-1,2) --(1,2);
\draw[thick] (1, 0) --(0, 1)--(1,2);
\draw[thick][dashed] (1,0) -- (2.4, 0);
\draw[thick][dashed] (1,2) -- (2.4, 2);
\draw[thick] (2.4, 0) -- (3.4,1)-- (2.4,2);
\draw[thick] (2.4,0) -- (4.4, 0) -- (3.4,1)-- (4.4,2) -- (2.4,2);
\draw[thick][dashed] (4.4,0) -- (5.8, 0);
\draw[thick][dashed] (4.4,2) -- (5.8, 2);
\draw[thick] (5.8, 0) -- (6.8,1)-- (5.8,2);
\draw[thick] (5.8,0) -- (7.8, 0) -- (6.8,1)-- (7.8,2) -- (5.8,2);
\draw[thick][dashed] (7.8,0) -- (8.8, 0);
\draw[thick][dashed] (7.8,2) -- (8.8, 2);

\draw[thick][dashed] (10,0) -- (11, 0);
\draw[thick][dashed] (10,2) -- (11, 2);
\draw[thick] (11, 0) -- (12,1)-- (11,2);
\draw[thick] (11,0) -- (13, 0) -- (12,1)-- (13,2) -- (11,2);
\draw[thick][dashed] (13,0) -- (14.4, 0);
\draw[thick][dashed] (13,2) -- (14.4, 2);
\draw[thick] (14.4, 0) -- (15.4,1)-- (14.4,2);

\end{tikzpicture}
\caption{$G \setminus (A \cup B)$ in Case 2}
\label{fig:case2}
\end{center}
\end{figure}

\begin{itemize}
\item[]\underline{Left child}: The left child of $\Sigma(A, B)$ is $\Sigma(A, B \cup \{a_j \})$. Since  $b_{(j-1,n-1)}$ and $c_{(j-1,n-1)}$ are pivots we add $b_{(j,1)}$ and $c_{(j,1)}$ to $A$ and the neighbors of $b_{(j,1)}$ and $c_{(j,1)}$ to $B$. Continuing to split prepare, we add the remaining vertices of $F_j$ to $A$ and $N(F_j)$ to $B$. 

\begin{itemize}
\item If $n \equiv 2 \pmod{3}$ the largest $k$ for which $\{b_{(j,k)},c_{(j,k)}\} \subseteq F_j$ is $n-1$. When $j=t$, $a_{t+1}$ is added to $B$ and we have a critical cell. When $j \neq t$, $a_{j+1}, b_{(j+1, 1)}$ and $c_{(j+1, 1)}$ are all added to $B$, so there are more pivots before we reach a split ready leaf. Continuing to split prepare, we add $H_{j+1}$ to $A$ and $N(H_{j+1})$ to $B$ and then go to Case 1 after increasing $j$ by one.
 
\item If $n \equiv 1 \pmod{3}$ the largest $k$ for which $\{b_{j,k},c_{j,k}\} \subseteq F_j$ is $n-3$. When $j = t$, $b_{(j,n-1)}$ and $c_{(j,n-1)}$ are pivots with $a_{t+1}$ as the associated matching vertex, so we add $a_{t+1}$ to $A$ and produce a critical cell. When $j \neq t$, we go to Case 2 after increasing $j$ by one.
\item If $n \equiv 0 \pmod{3}$ the largest $k$ for which $\{b_{j,k},c_{j,k}\} \subseteq F_j$ is $n-2$. When $j = t$, $a_{t+1}$ has no neighbors in $G \setminus(A \cup B)$, so we do not get a critical cell. When $j \neq t$, we go to Case 1 after increasing $j$ by one.

\end{itemize}

\item[]\underline{Right child}: The right child in Case 2 is the same as the right child in Case 1.\\
\end{itemize}

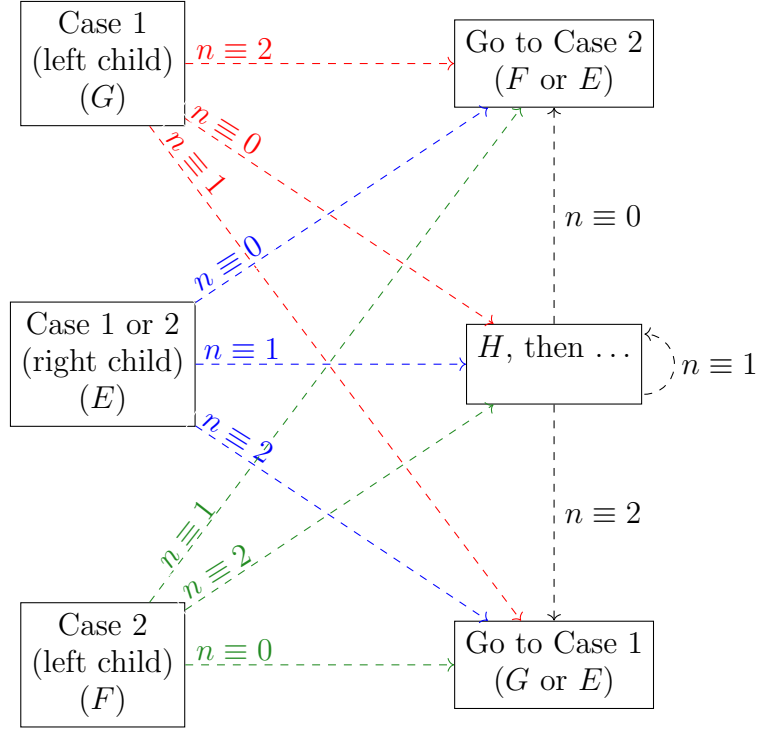
\begin{figure}[!htb] 
\begin{center}
\begin{tikzpicture}
\node (C) [draw, align=center] at (0,0) {Case 1\\ (left child)\\ ($G$)};
\node (A) [draw, align=center] at (0,-4) {Case 1 or 2\\ (right child)\\ ($E$)};
\node (B) [draw, align=center] at (0,-8) {Case 2\\ (left child)\\ ($F$)};

\node (1) [draw, align=center] at (6,0) {Go to Case 2\\ ($F$ or $E$)};
\node (2) [draw, align=center] at (6,-8) {Go to Case 1\\ ($G$ or $E$)};
\node (3) [draw, align=center] at (6,-4) {$H$, then $\ldots$\\ };

\draw[dashed,->, red] (C) -- node[above]{ } (1);
\draw[dashed,->, red] (C) -- node[right]{ } (2);
\draw[dashed,->, red] (C) -- node[right]{ } (3);

\node[draw = white, rotate = -50] at (1.25, -1.3) {\color{red}{$n \equiv 1$}};
\node at (1.75, 0.2) {\color{red}{$n \equiv 2$}};
\node[draw = white, rotate = -32]  at (1.65, -.85) {\color{red}{$n \equiv 0$}};

\draw[dashed,->, blue] (A) -- node[left] { } (1);
\draw[dashed,->, blue] (A) -- node[left]{} (2);
\draw[dashed,->, blue] (A) -- node[left]{} (3);

\node[draw = white, rotate = 32] at (1.65, -2.65) {\color{blue}{$n \equiv 0$}};
\node at (1.85, -3.8) {\color{blue}{$n \equiv 1$}};
\node[draw = white, rotate = -32]  at (1.8, -4.95) {\color{blue}{$n \equiv 2$}};

\draw[dashed,->, green] (B) -- node[right, above] { } (1);
\draw[dashed,->, green] (B) -- node[right]{ } (2);
\draw[dashed,->, green] (B) -- node[below]{ } (3);

\node[draw = white, rotate = 52] at (1.1, -6.25) {\color{green}{$n \equiv 1$}};
\node[draw = white, rotate = 32] at (1.5, -6.75) {\color{green}{$n \equiv 2$}};
\node at (1.75, -7.8) {\color{green}{$n \equiv 0$}};

\draw[dashed,->, black] (3) -- node[right] { $n \equiv 0$} (1);
\draw[dashed,->, black] (3) -- node[right]{ $n \equiv 2$ } (2);
\draw[dashed, ->] (7.2, -4.4) arc (-90:90:.4cm);

\node at (8.2, -4) {{$n \equiv 1$}};

\end{tikzpicture}

\end{center}
\caption{In the analysis in the proof of Theorem \ref{2ngons}, for any leaf $\Sigma(A, B)$ we represent $A$ as a string of $t$ letters $E, F, G$, and $H$. The chart above aids in determining which strings are possible.}
\label{fig:analysis}
\end{figure}

\noindent {\bf Analysis:} 
We will now determine the minimum and maximum number of vertices in the critical cells.
By our algorithm, for any leaf $\Sigma(A,B)$ of the matching tree, $A = X_1 \cup X_2 \cup \dots \cup X_t$ or $A = X_1 \cup X_2 \cup \dots \cup X_t \cup \{a_{t+1}\}$ where each $X_i$ is either $E_i$, $F_i$, $G_i$, or $H_i$. For ease of notation, we will omit the subscripts as well as $a_{t+1}$ and represent $A$ as a string of length $t$ letters $E$, $F$, $G$, and $H$. For example, if $A = E_1 \cup F_2 \cup E_3 \cup \{a_{t+1}\}$, we represent it by $EFE$, and we say that ``$A$ begins with $E$,'' ``$E$ is followed by $F$,'' and so on.
 Whether or not $a_{t+1} \in A$ depends on the final letter of the string. 

We use the casework above to determine which of these strings correspond to critical cells. Much of the information below is also presented in Figure \ref{fig:analysis}.\\

\noindent If $n \equiv 0 \bmod 3$:
\begin{itemize}
\item $A$ must begin with $E$ or $G$, since we begin in Case 1.
\item $E$ adds $\frac {2n}3-1$ vertices and is followed by $E$ or $F$.
\item $F$ adds $\frac {2n}3$ vertices and is followed by $E$ or $G$.
\item $G$ adds $\frac {2n}3$ vertices and is followed by $H$.
\item $H$ adds $\frac {2n}3-2$ vertices and is followed by $E$ or $F$.
\item $A$ must end with $E$, $G$, or $H$. We cannot end in $F$ by the left child of Case 2. 
\item If $A$ ends with $E$ or $H$, add one vertex (for $a_{t+1}$).
\end{itemize}

Notice that while $H$ adds the smallest number of vertices, it must be preceded by $G$, so the substring $GH$ adds the same number of vertices as $EE$, namely $\frac{4n}{3} -2$. Additionally, regardless of whether we end with $E$, $G$, or $H$, the final two letter substring adds $\frac{4n}{3} -2$ vertices to $A$ because of $a_{j+1}$. Therefore, the smallest critical cells have $\frac{2nt}{3} - t + 1$ elements and occur when $A$ doesn't include any $F$'s (for example, when $A$ is a string of $E$'s). 
Thus, $d^{\min} = \frac{2nt}{3}-t$.

The largest critical cells occur when $A$ contains a maximum number of $F$'s. 
When $t$ is odd, the largest critical cells occur when $A$ is the string $EFEF \cdots FE$. When $t$ is even, the largest critical cells occur when $A$ is the same string but with one additional $E$.
Therefore in the largest critical cells, $A$ contains $\frac{2nt}3 - \lfloor  \frac{t}{2} \rfloor$ vertices. Thus, $d^{\max} = \frac{2nt}3 - \lfloor \frac{t}{2} \rfloor-1$.\\

\noindent If $n \equiv 1 \bmod 3$:
\begin{itemize}
\item $A$ must begin with $E$ or $G$.
\item $E$ adds $\frac {2n+1}3$ vertices and is followed by $H$.
\item $F$ adds $\frac {2n+1}3-1$ vertices and is followed by $E$ or $F$.
\item $G$ adds $\frac {2n+1}3-1$ vertices and is followed by $E$ or $G$.
\item $H$ adds $\frac {2n+1}3-1$ vertices and is followed by $H$.
\item $A$ must end with $E$, $F$ or $H$. We cannot end in $G$ by the left child of Case 1.

\item If $A$ ends with $F$, add one vertex (for $a_{t+1}$).
\end{itemize}

It is impossible for $A$ to include an $F$ and $A$ always contains exactly one $E$. Every critical cell has $\frac {2nt+t}3-t+1$ elements. Thus $d^{\min} = d^{\max} = \frac{2nt+t}3 - t$.\\

\noindent If $n \equiv 2 \bmod 3$:
\begin{itemize}
\item $A$ must begin with $E$ or $G$.
\item $E$ adds $\frac {2n-1}3$ vertices and is followed by $E$ or $G$.
\item $F$ adds $\frac {2n-1}3+1$ vertices and is followed by $H$.
\item $G$ adds $\frac {2n-1}3-1$ vertices and is followed by $E$ or $F$.
\item $H$ adds $\frac {2n-1}3-1$ vertices and is followed by $E$ or $G$.
\item $A$ must end with $F$ or $G$. We cannot end in $E$ by the right child of Case 1 or Case 2 or $H$ by the left child of Case 2. 
\item If $A$ ends with $G$, add one vertex (for $a_{t+1}$).
\end{itemize}

Since $G$ and $H$ add the smallest number of vertices, but $H$ must be preceded by $F$, the smallest critical cells occur when $A$ contains the maximum number of $G$'s possible. When $t$ is odd, the smallest critical cells occur when $A$ is the string $GEGEG \cdots EG$. When $t$ is even, the smallest critical cells occur when $A$ is the string $EGEG \cdots EG$. Therefore the smallest critical cells have $\frac{2nt-t}{3}- \lfloor \frac {t+1}2 \rfloor + 1$ elements and so $d^{\min} = \frac{2nt-t}{3}- \lfloor \frac {t+1}2 \rfloor$.

The largest critical cells have  $\frac{2nt-t}3$ elements. Note that while $F$ adds the largest number of vertices, it must be followed by $H$ and preceded by $G$, and the substring $GFH$ adds fewer vertices than $EEE$. So the largest critical cells occur when $A$ is the string $EEE \cdots EGF$ or  $EEE \cdots EG$. In both cases, $A$ has $\frac{2nt-t}3$ vertices so $d^{\max} = \frac{2nt-t}{3} -1$.

\end{proof}

\begin{example}
Example~\ref{ex:mta} gives an application of Theorem~\ref{2ngons} in the case where $n=3$ and $t=2$. In this example, the matching tree has two nonempty leaf nodes, 
\[\Sigma( \{1,6,11\}, \{ 2,3,4,5,7,8,9,10\}) \text{ and } \Sigma( \{4,5, 11 \}, \{1, 2, 3 , 6, 7, 8, 9, 10\}).\]
Thus the two possibilities for $A$ are $\{1,6, 11\}$, represented by the string $EE$, and $\{4,5, 11 \}$, represented by the string $GH$. Both critical cells have size two, consistent with the fact that $d^{\max}=\frac{2(3)(2)}{3}-2  =d^{\min}$. 
\end{example}

\begin{corollary}
\label{cor:nequiv1}
When $n \equiv 1 \pmod 3$ and $t \geq 2$,
the matching complex of $t$ $2n$-gons is homotopy equivalent to a wedge of $t$ spheres of dimension $\frac{2nt +t}3 - t$.
\end{corollary}

\begin{proof}
When $n \equiv 1 \pmod 3$ in Theorem~\ref{2ngons}, $d^{\max} = d^{\min} = \frac{2nt +t}3 - t$.
It follows from Theorem~\ref{thm:main_mta} that the homotopy type is a wedge of spheres.
Recall from the proof of Theorem \ref{2ngons} that $A$ can be thought of as a string of letters of length $t$. Furthermore, by the analysis in the $n \equiv 1$ case, there must be exactly one $E$ in this string, which is preceded by $G$'s and followed by $H$'s. Since the number of critical cells is equal to the number of possible strings of length $t$ of this form, there are $t$ critical cells.
\end{proof}

When $n = 3$ in Figure~\ref{original} we have a $1 \times 1 \times t$ honeycomb graph.
Therefore Theorem~\ref{2ngons} provides partial results for the honeycomb question that arises in~\cite{Jakob}.

\begin{corollary}
\label{1by1bym}
The matching complex of a $1 \times 1\times t$ honeycomb graph is at least $(t-1)$-connected.
\end{corollary}
\begin{proof}
When $t\geq 2$, this follows immediately from the $n=3$ case of Theorem \ref{2ngons}. When $t = 1$, a $1 \times 1 \times t$ honeycomb graph is the cycle $C_6$, and $M(C_6) \simeq S^1 \vee S^1$ \cite[Proposition 4.6]{Kozlov_Trees}. 
\end{proof}

We will show in Theorem~\ref{2by1bym} that the $2 \times 1 \times t$ honeycomb is $(2t-1)$-connected. This sharpens previous connectivity bounds due to Barmak \cite{Barmak} and Engstr\"om \cite{Engstrom_ClawFree}. In these papers, the authors independently explored the connectivity of the independence complexes of claw-free graphs. The {\em claw graph} is the complete bipartite graph $K_{1, 3}$. A graph is {\em claw-free} if there are no induced subgraphs which are claw graphs.

As the matching complex of a graph is the independence complex of the line graph, and all line graphs are claw-free,  
Barmak's and Engstr\"om's results can be used to obtain connectivity bounds for matching complexes. For the $2 \times 1 \times t$ honeycomb, our bound is sharper than the bounds provided by Barmak \cite[Theorem 5.5]{Barmak} and Engstr\"om \cite[Theorem 3.2]{Engstrom_ClawFree}.

We start with a weak connectivity bound for any $r \times s \times t$ honeycomb graph based on the values of $r$, $s$, and $t$. 
Notice that a honeycomb and its line graph have $r+s-1$ rows of hexagons.

\begin{remark}\label{lem:numhex}
The number of hexagons in the $i^{th}$ row from the top of an $r \times s \times t$ honeycomb graph (or line graph) is $\rho(i)$, where

\begin{equation*}
    \rho(i) = \begin{cases}
        t+i-1 & \text{if } 1  \le i\le \min(r,s) \\
        t + \min(r,s) - 1 & \text{if } \min(r,s) < i \le \max(r,s)\\
        t + r + s - 1 - i &\text{if } \max(r,s) <i \le r+s - 1
        \end{cases}
\end{equation*} \textbf{}
\end{remark}

\begin{figure}[!htb] 
\begin{center}
\begin{tikzpicture}[scale = .85]

\draw[thick] (2.5,2.5) -- (7.5,2.5);
\draw[thick] (1.5,1.5) -- (8.5,1.5);
\draw[thick] (.5,.5) -- (9.5,.5);
\draw[thick] (.5,-.5) -- (9.5,-.5);
\draw[thick] (1.5,-1.5) -- (8.5,-1.5);
\draw[thick] (2.5,-2.5) -- (7.5,-2.5);

\draw[thick] (0,0) -- (2.5,-2.5);
\draw[thick] (1,1) -- (4.5,-2.5);
\draw[thick] (2,2) -- (6.5,-2.5);
\draw[thick] (3.5,2.5) -- (8,-2);
\draw[thick] (5.5,2.5) -- (9,-1);
\draw[thick] (7.5,2.5) -- (10,0);

\draw[thick] (0,0) -- (2.5,2.5);
\draw[thick] (1,-1) -- (4.5,2.5);
\draw[thick] (2,-2) -- (6.5,2.5);
\draw[thick] (3.5,-2.5) -- (8,2);
\draw[thick] (5.5,-2.5) -- (9,1);
\draw[thick] (7.5,-2.5) -- (10,0);

\filldraw (2,2) circle (0.15cm); 
\filldraw (4,2) circle (0.15cm); 
\filldraw (6,2) circle (0.15cm); 
\filldraw (8,2) circle (0.15cm); 

\filldraw (0,0) circle (0.15cm); 
\filldraw (2,0) circle (0.15cm); 
\filldraw (4,0) circle (0.15cm);
\filldraw (6,0) circle (0.15cm);
\filldraw (8,0) circle (0.15cm);
\filldraw (10,0) circle (0.15cm);

\filldraw (2,-2) circle (0.15cm);
\filldraw (4,-2) circle (0.15cm);
\filldraw (6,-2) circle (0.15cm);
\filldraw (8,-2) circle (0.15cm);

\end{tikzpicture}
\caption{The line graph of a $3 \times 3 \times 3$ honeycomb}
\label{fig:333}
\end{center}
\end{figure}
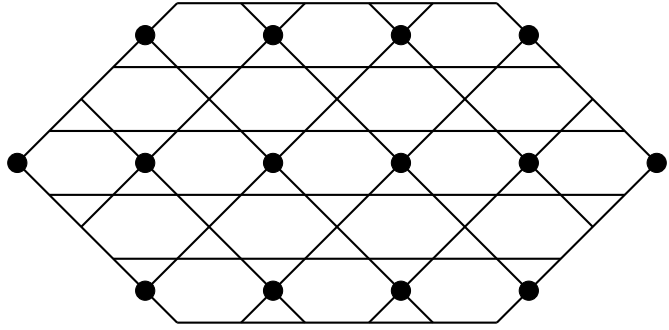

\begin{prop} \label{prop:genhex}
The matching complex of an $r\times s\times t$ honeycomb graph $G$ is at least
$k$-connected, where 
$$k = 
t \left\lceil \frac{r+s-1}2 \right\rceil +\left\lceil \dfrac{ rs }{2} \right\rceil - 2.
$$
\end{prop}
\begin{proof}
Consider the subset $X \subset V(G)$ consisting of the center vertices of each hexagon in every $i$th row for all $i \in \{1, 3, 5, \ldots, r+s-1\}$ (see Figure~\ref{fig:333}). Each row contributes one more vertex than the number of hexagons in the row. 
We see that $|X| = \sum\limits_{i=1}^{\lceil \frac{r+s-1}2 \rceil} (\rho(2i-1)+1)$, 
where $\rho(i)$ is defined in Remark~\ref{lem:numhex}.
Since the neighbors of each vertex are disjoint, we can apply Lemma~\ref{lem:buzzerbeater} to conclude that all critical cells have at least $\sum\limits_{i=1}^{\lceil \frac{r+s-1}2 \rceil} (\rho(2i-1)+1)$ elements. We claim that
$$\sum\limits_{i=1}^{\lceil \frac{r+s-1}2 \rceil} (\rho(2i-1)+1) = t \left\lceil \frac{r+s-1}2 \right\rceil +\left\lceil \dfrac{rs }{2} \right\rceil$$

We will prove this when $r$ and $s$ are odd. The general case is similar. By the definition of $\rho$, 

\begin{eqnarray*}
& & \sum\limits_{i=1}^{\lceil \frac{r+s-1}2 \rceil} (\rho(2i-1)+1)\\
& = & 
\sum\limits_{i=1}^{\lceil \frac{r+s-1}2 \rceil} t + 
\sum\limits_{i =1}^{ \frac{ \min(r, s) +1}{2} }  (2i-1) + 
\sum\limits_{i =\frac{ \min(r, s) +1}{2} + 1}^{ \frac{ \max(r, s) +1}{2} }  \min(r, s) +
\sum\limits_{i =\frac{ \max(r, s) +1}{2} + 1 }^{ \lceil \frac{r+s-1}2 \rceil} (r + s  - 2i+1). \\
\end{eqnarray*}

Since $\sum\limits_{i =\frac{ \max(r, s) +1}{2} + 1 }^{ \lceil \frac{r+s-1}2 \rceil} (r + s  - 2i+1) = 
\sum\limits_{i =1 }^{ \frac{\min(r, s)- 1}2 } (\min(r, s) - 2i)$, we see that
\begin{eqnarray*}
 \sum\limits_{i=1}^{\lceil \frac{r+s-1}2 \rceil} (\rho(2i-1)+1)
 &=& 
\sum\limits_{i=1}^{\lceil \frac{r+s-1}2 \rceil} t + 
\left( \sum\limits_{i =1 }^{ \frac{ \max(r, s) +1}{2}}\min(r, s) \right)
 - \frac{\min(r, s)- 1}2 \\
& =&
 t \left\lceil \frac{r+s-1}2 \right\rceil +\left\lceil \dfrac{ \max(r, s) \min(r, s) }{2} \right\rceil
\end{eqnarray*}

The connectivity bound follows.
\end{proof}
 For a $2 \times 1 \times t$ honeycomb graph, we can greatly improve this connectivity bound and obtain 
both upper and lower bounds for the sizes of the critical cells.

The following lemma describes the outcome to applying the Matching Tree Algorithm to a configuration that frequently occurs in the proof of Theorem~\ref{2by1bym}. 
\begin{lemma}
\label{lem:snowman}

Let $\Sigma(A, B)$ be a split ready leaf associated to the configuration with a hexagon and a triangle that share an edge. Then, after completing the Matching Tree Algorithm, the two leaf nodes 
$\Sigma(\tilde{A}, \tilde{B})$, $\Sigma(\hat{A}, \hat{B})$
with lowest common ancestor $\Sigma(A, B)$ 
have $|\tilde{A}| = |\hat{A}|= |A| + 2$.
\end{lemma}

\begin{proof}

Choose the indicated vertex of the hexagon as our next splitting vertex, as pictured to the right\footnote{For this configuration, the choice of splitting vertex does not affect the size or quantity of critical cells, but this is not true for all configurations.}. 
Then splitting produces two leaves, $\Sigma(A, B \cup \{v \})$ and $\Sigma(A \cup \{v\}, B \cup N(v))$. 

\noindent
\begin{minipage}[t]{.775\textwidth}
\vspace{-.75cm}
\flushleft
When split preparing the left child $\Sigma(A, B \cup \{v \})$ we add two matching vertices to $A$. When split preparing the right child $\Sigma(A \cup \{v\}, B \cup N(v))$, we add one matching vertex to $A \cup \{v\}$. Thus, the two children of $\Sigma(A, B)$ have critical cells with $|A| + 2$ elements. 
\end{minipage}
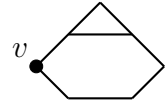
\begin{minipage}{.2\textwidth}

\begin{tikzpicture}[scale = .85]

 \filldraw[color=white,draw = white] (12,2) circle (0.05cm); 
  
\filldraw (11,0.5) circle (0.1cm); 

\draw[thick] (11.5,1) -- (12.5,1);
\draw[thick] (11.5,0) -- (12.5,0);
\draw[thick] (11,.5) -- (11.5,0);
\draw[thick] (12,1.5) -- (13,.5);
\draw[thick] (11,.5) -- (12,1.5);
\draw[thick] (12.5,0) -- (13,.5);

\node at (10.75, .75) {$v$};

\end{tikzpicture}

\end{minipage}

\end{proof}

\begin{theorem}
\label{2by1bym}
Let $t \geq 1$ and $n > 2$. 
Define
$d^{\min} = 2t$
and
$d^{\max} = \frac{7t}{3}+1$.The matching complex of a $2 \times 1 \times t$ honeycomb graph is homotopy equivalent to a space with no $d$-dimensional cells, where $0 < d < d^{\min}$ and $d \ge d^{\max}$. Further, the connectivity is at least $2t -1$.
\end{theorem}

\begin{proof}
A $2\times 1 \times t$ honeycomb is made up of two rows of $t$ hexagons with the line graph shown in Figures~\ref{2m1} and \ref{fig:steps}. 
The proof of this theorem is similar to the proof of Theorem \ref{2ngons}. We will again use the Matching Tree Algorithm with specific choices of splitting vertices. However, unlike in Theorem \ref{2ngons}, at each step the choice of the next splitting vertex depends on the current configuration. We apply the following algorithm:

For each split ready leaf node $\Sigma(A, B)$:
\begin{itemize}
\item[] Step 1: Choose the appropriate splitting vertex $v$ depending on the current configuration. This produces two leaves, $\Sigma(A, B \cup \{v\})$ and $\Sigma(A \cup \{v\}, B \cup N(v))$.
\item[] Step 2: Split prepare each leaf. 
\end{itemize}

At Step 1 of this algorithm, there will be six possible configurations distinguished by the leftmost part of the graph up to vertical reflection, as seen in Figure~\ref{2m1} and Figures~\ref{2m2} -- ~\ref{2m6}.
For each of these cases, we mark the splitting vertex in bold. Note that our choices of splitting vertices is somewhat arbitrary. They are intended to minimize casework, but it is possible that our choices could be further optimized. 

Let \Thex~:=\Thex$(\Sigma(A,B))$ be the number of complete hexagons in $G \setminus (A \cup B)$. Note that \Thex$(\Sigma(\emptyset,\emptyset)) = 2t$ and \Thex$(\Sigma(A,B)) = 0$ when $\Sigma(A,B)$ corresponds to a critical cell. 

As we move down our matching tree, we add vertices to the left set of the matching tree nodes. When $\Sigma(A',B')$ is a descendant of $\Sigma(A,B)$ we will abuse notation by saying that $|A|$ ``increases'' by $|A'| - |A|$ as we go from $\Sigma(A,B)$ to $\Sigma(A',B')$. As we proceed through the algorithm, we will keep track of $|A|$, which will increase as we move down the matching tree and \Thex, which will decrease as we move down the matching tree. We calculate $d^{\min}$ and $d^{\max}$ by observing the relationship between the changes to $|A|$ and \Thex.\\

\noindent {\bf Case 1:}

\begin{figure}[!htb] 
\begin{center}
\begin{tikzpicture}[scale = .85]

\filldraw (.5,1) circle (0.1cm); 

\draw[thick] (.5,2) -- (5.5,2);
\draw[thick][dashed] (5.5,2) -- (6,2);
\draw[thick] (.5,1) -- (6.5,1);
\draw[thick][dashed] (6.5,1) -- (7,1);
\draw[thick] (1.5,0) -- (7.5,0);
\draw[thick][dashed] (7.5,0) -- (8,0);

\draw[thick] (0,1.5) -- (1.5,0);
\draw[thick] (1.5,2) -- (3.5,0);
\draw[thick] (3.5,2) -- (5.5,0);
\draw[thick] (5.5,2) -- (7.5,0);

\draw[thick] (0,1.5)--(.5,2);
\draw[thick] (1,.5)--(2.5,2);
\draw[thick] (2.5,0)--(4.5,2);
\draw[thick] (4.5,0)--(6,1.5);
\draw[thick][dashed] (6,1.5)--(6.25,1.75);
\draw[thick] (6.5,0)--(7,.5);
\draw[thick][dashed] (7,.5)--(7.25,.75);

\end{tikzpicture}
\caption{Case 1}
\label{2m1}
\end{center}
\end{figure}

\begin{itemize}
\item For the left child ($\Sigma(A,B \cup \{v\})$):
\begin{itemize}
\begin{minipage}{.65\textwidth}
\item If \Thex~ = 2, we have the configuration pictured right. We add one matching vertex to $A$ when split preparing, and the next configuration is one hexagon and a triangle that share an edge. By Lemma~\ref{lem:snowman}, we get two critical cells. For each of the cells, $|A|$ increases by 3 while \Thex~ decreases by 2. 
\end{minipage}\hfill \hspace{.5cm}
\begin{minipage}{.25\textwidth}

\begin{tikzpicture}[scale = .85]

\filldraw (10.5,1) circle (0.1cm); 
\draw[thick] (10.5,2) -- (11.5,2);
\draw[thick] (10.5,1) -- (12.5,1);
\draw[thick] (11.5,0) -- (12.5,0);
\draw[thick] (10,1.5) -- (11.5,0);
\draw[thick] (11.5,2) -- (13,.5);
\draw[thick] (10,1.5) -- (10.5,2);
\draw[thick] (11,.5) -- (12,1.5);
\draw[thick] (12.5,0) -- (13,.5);

\end{tikzpicture}

\end{minipage}

\item If \Thex~ $\geq 3$, we add one matching vertex to $A$, decrease \Thex~ by one, and repeat Case 1, as a vertical reflection. The argument is analogous by symmetry. From now on, we will take these vertical symmetries for granted without further comment.

\end{itemize}

\item For the right child ($\Sigma(A \cup \{v\},B \cup N(v))$):

\begin{itemize}
\item If \Thex~ $\leq 5$, we get no critical cell.
\item If \Thex~ $=6$, we add five vertices to $A$ when split preparing, including $v$ (see Figure~\ref{fig:steps}). We are left with one hexagon with two attached triangles and we choose our next splitting vertex to be any of the vertices on both the hexagon and a triangle. Then, we get one critical cell, adding two more vertices to $A$: either two matching vertices or one splitting and one matching vertex. Thus, for each critical cell, $|A|$ increases by 7 while \Thex~ decreases by 6. 
\item If \Thex~ $>6$, we add five vertices to $A$ including $v$, decrease \Thex~ by five, and go to Case 2.\\

\end{itemize}
\end{itemize}

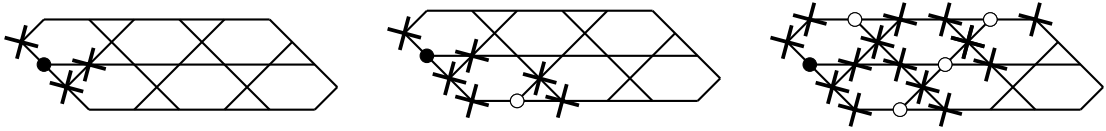
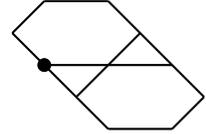
\begin{figure}[!htb] 
\begin{minipage}{.3\textwidth}
\begin{tikzpicture}[scale = .6]

\filldraw (.5,1) circle (0.15cm); 
\draw (0,1.5) node[cross] {};
\draw (1.5,1) node[cross] {};
\draw (1, .5) node[cross] {};

\draw[thick] (.5,2) -- (5.5,2);

\draw[thick] (.5,1) -- (6.5,1);

\draw[thick] (1.5,0) -- (6.5,0);

\draw[thick] (0,1.5) -- (1.5,0);
\draw[thick] (1.5,2) -- (3.5,0);
\draw[thick] (3.5,2) -- (5.5,0);
\draw[thick] (5.5,2) -- (7, .5);

\draw[thick] (0,1.5)--(.5,2);
\draw[thick] (1,.5)--(2.5,2);
\draw[thick] (2.5,0)--(4.5,2);
\draw[thick] (4.5,0)--(6,1.5);

\draw[thick] (6.5,0)--(7,.5);

\end{tikzpicture}
\end{minipage}
\begin{minipage}{.3\textwidth}
\begin{tikzpicture}[scale = .6]

\filldraw (.5,1) circle (0.15cm); 
\draw (0,1.5) node[cross] {};
\draw (1.5,1) node[cross] {};
\draw (1, .5) node[cross] {};

\filldraw[fill=white] (2.5,0) circle (0.15cm); 
\draw (1.5,0) node[cross] {};
\draw (3.5,0) node[cross] {};
\draw (3,.5) node[cross] {};

\draw[thick] (.5,2) -- (5.5,2);

\draw[thick] (.5,1) -- (6.5,1);

\draw[thick] (1.5,0) -- (6.5,0);

\draw[thick] (0,1.5) -- (1.5,0);
\draw[thick] (1.5,2) -- (3.5,0);
\draw[thick] (3.5,2) -- (5.5,0);
\draw[thick] (5.5,2) -- (7, .5);

\draw[thick] (0,1.5)--(.5,2);
\draw[thick] (1,.5)--(2.5,2);
\draw[thick] (2.5,0)--(4.5,2);
\draw[thick] (4.5,0)--(6,1.5);

\draw[thick] (6.5,0)--(7,.5);
\filldraw[fill=white] (2.5,0) circle (0.15cm); 

\end{tikzpicture}
\end{minipage}
\begin{minipage}{.3\textwidth}
\begin{tikzpicture}[scale = .6]

\filldraw (.5,1) circle (0.15cm); 
\draw (0,1.5) node[cross] {};
\draw (1.5,1) node[cross] {};
\draw (1, .5) node[cross] {};

\filldraw[fill=white] (2.5,0) circle (0.15cm); 
\draw (1.5,0) node[cross] {};
\draw (3.5,0) node[cross] {};
\draw (3,.5) node[cross] {};

\filldraw[fill=white] (1.5,2) circle (0.15cm); 
\draw (.5,2) node[cross] {};
\draw (2.5,2) node[cross] {};
\draw (2,1.5) node[cross] {};

\filldraw[fill=white] (3.5,1) circle (0.15cm); 
\draw (2.5,1) node[cross] {};
\draw (4.5,1) node[cross] {};
\draw (4,1.5) node[cross] {};

\filldraw[fill=white] (4.5,2) circle (0.15cm); 
\draw (3.5,2) node[cross] {};
\draw (5.5,2) node[cross] {};
\draw (4,1.5) node[cross] {};

\draw[thick] (.5,2) -- (5.5,2);

\draw[thick] (.5,1) -- (6.5,1);

\draw[thick] (1.5,0) -- (6.5,0);

\draw[thick] (0,1.5) -- (1.5,0);
\draw[thick] (1.5,2) -- (3.5,0);
\draw[thick] (3.5,2) -- (5.5,0);
\draw[thick] (5.5,2) -- (7, .5);

\draw[thick] (0,1.5)--(.5,2);
\draw[thick] (1,.5)--(2.5,2);
\draw[thick] (2.5,0)--(4.5,2);
\draw[thick] (4.5,0)--(6,1.5);

\draw[thick] (6.5,0)--(7,.5);

\filldraw[fill=white] (2.5,0) circle (0.15cm); 
\filldraw[fill=white] (1.5,2) circle (0.15cm); 
\filldraw[fill=white] (3.5,1) circle (0.15cm); 
\filldraw[fill=white] (4.5,2) circle (0.15cm); 

\end{tikzpicture}

\end{minipage}
\caption{Case 1 for \Thex~$=6$: The vertices in $A$ are represented with a circle, where an open circle represents a matching vertex. The vertices in $B$ are represented with an ``$X$''. The leftmost image shows the splitting vertex in $A$ and its neighborhood in $B$. Continuing to split prepare, we see there are two possible matching vertices. The middle picture shows the result of choosing one of them. We continue this process until there are no more matching vertices, shown in the rightmost image. We are left with a hexagon with two attached triangles. }
\label{fig:steps}
\end{figure}

\noindent {\bf Case 2:}

\begin{figure}[!htb] 
\begin{center}
\begin{tikzpicture}[scale = .85]

\filldraw (-.5,2) circle (0.1cm); 

\draw[thick] (-.5,2) -- (5.5,2);
\draw[thick][dashed] (5.5,2) -- (6,2);
\draw[thick] (.5,1) -- (6.5,1);
\draw[thick][dashed] (6.5,1) -- (7,1);
\draw[thick] (1.5,0) -- (7.5,0);
\draw[thick][dashed] (7.5,0) -- (8,0);

\draw[thick] (-.5,2) -- (1.5,0);
\draw[thick] (1.5,2) -- (3.5,0);
\draw[thick] (3.5,2) -- (5.5,0);
\draw[thick] (5.5,2) -- (7.5,0);

\draw[thick] (0,1.5)--(.5,2);
\draw[thick] (1,.5)--(2.5,2);
\draw[thick] (2.5,0)--(4.5,2);
\draw[thick] (4.5,0)--(6,1.5);
\draw[thick][dashed] (6,1.5)--(6.25,1.75);
\draw[thick] (6.5,0)--(7,.5);
\draw[thick][dashed] (7,.5)--(7.25,.75);

\end{tikzpicture}
\caption{Case 2}
\label{2m2}
\end{center}
\end{figure}

\begin{itemize}
\item For the left child ($\Sigma(A,B \cup \{v\})$):
\begin{itemize}
\item We do not add anything to $A$ or change \Thex. Go back to Case 1. 

\end{itemize}

\item For the right child ($\Sigma(A \cup \{v\},B \cup N(v))$):

\begin{itemize}
\item If \Thex~$ = 2$, we get a critical cell after adding four vertices to $A$, including $v$. 

\item If \Thex~ $\geq 3$, we only add $v$ to $A$, decrease \Thex~ by one, then go to Case 3.\\ 
\end{itemize}
\end{itemize}

\noindent {\bf Case 3:}

\begin{figure}[!htb] 
\begin{center}
\begin{tikzpicture}[scale = .85]

\filldraw (1.5,1) circle (0.1cm); 

\draw[thick] (1.5,2) -- (5.5,2);
\draw[thick][dashed] (5.5,2) -- (6,2);
\draw[thick] (.5,1) -- (6.5,1);
\draw[thick][dashed] (6.5,1) -- (7,1);
\draw[thick] (1.5,0) -- (7.5,0);
\draw[thick][dashed] (7.5,0) -- (8,0);

\draw[thick] (.5,1) -- (1.5,0);
\draw[thick] (1.5,2) -- (3.5,0);
\draw[thick] (3.5,2) -- (5.5,0);
\draw[thick] (5.5,2) -- (7.5,0);

\draw[thick] (1,.5)--(2.5,2);
\draw[thick] (2.5,0)--(4.5,2);
\draw[thick] (4.5,0)--(6,1.5);
\draw[thick][dashed] (6,1.5)--(6.25,1.75);
\draw[thick] (6.5,0)--(7,.5);
\draw[thick][dashed] (7,.5)--(7.25,.75);

\end{tikzpicture}
\caption{Case 3}
\label{2m3}
\end{center}
\end{figure}

\begin{itemize}
\item For the left child ($\Sigma(A,B \cup \{v\})$):
\begin{itemize}
\item If \Thex~ = 2, we get a critical cell after adding four matching vertices to $A$. 
\item If \Thex~ $\geq 3$, we add one matching vertex to $A$, decrease \Thex~ by one, and go to Case 4. 
\end{itemize}

\item For the right child ($\Sigma(A \cup \{v\},B \cup N(v))$):

\begin{itemize}
\item If \Thex~ $= 2$, we get a critical cell after adding four vertices to $A$, including $v$. 
\item If \Thex~ $= 3$, we get a critical cell after adding five vertices to $A$, including $v$. 
\item If \Thex~ $= 4$, we get a critical cell after adding six vertices to $A$, including $v$.
\item If \Thex~ $\geq 5$, we add three vertices to $A$, including $v$, and decrease \Thex~ by three before repeating Case 3.\\
\end{itemize}
\end{itemize}

\noindent {\bf Case 4:}

\begin{figure}[!htb] 
\begin{center}
\begin{tikzpicture}[scale = .85]

\filldraw (.5,0) circle (0.1cm); 

\draw[thick] (-.5,2) -- (5.5,2);
\draw[thick][dashed] (5.5,2) -- (6,2);
\draw[thick] (.5,1) -- (6.5,1);
\draw[thick][dashed] (6.5,1) -- (7,1);
\draw[thick] (.5,0) -- (7.5,0);
\draw[thick][dashed] (7.5,0) -- (8,0);

\draw[thick] (-.5,2) -- (1.5,0);
\draw[thick] (1.5,2) -- (3.5,0);
\draw[thick] (3.5,2) -- (5.5,0);
\draw[thick] (5.5,2) -- (7.5,0);

\draw[thick] (0,1.5)--(.5,2);
\draw[thick] (.5,0)--(2.5,2);
\draw[thick] (2.5,0)--(4.5,2);
\draw[thick] (4.5,0)--(6,1.5);
\draw[thick][dashed] (6,1.5)--(6.25,1.75);
\draw[thick] (6.5,0)--(7,.5);
\draw[thick][dashed] (7,.5)--(7.25,.75);

\end{tikzpicture}
\caption{Case 4}
\label{2m4}
\end{center}
\end{figure}

\begin{itemize}
\item For the left child ($\Sigma(A,B \cup \{v\})$):
\begin{itemize}
\item We do not add anything to $A$ or change \Thex. Go back to Case 2.
\end{itemize}

\item For the right child ($\Sigma(A \cup \{v\},B \cup N(v))$):

\begin{itemize}
\item If \Thex~ $=2$, we add two vertices to $A$ including $v$, and are left with a hexagon and a triangle that share an edge. By Lemma~\ref{lem:snowman}, we get two critical cells. For each of the cells, $|A|$ increases by 4. 
\item If \Thex~ $=3$, we add three vertices to $A$ including $v$, and are left with a hexagon and a triangle that share an edge. By Lemma~\ref{lem:snowman}, we get two critical cells. For each of the cells, $|A|$ increases by 5. 
\item If \Thex~ $\geq 4$, we only add $v$ to $A$, decrease \Thex~ by one, then go to Case 5.\\
\end{itemize}
\end{itemize}

\noindent {\bf Case 5:}

\begin{figure}[!htb] 
\begin{center}
\begin{tikzpicture}[scale = .85]

\filldraw (0,1.5) circle (0.1cm); 

\draw[thick] (-.5,2) -- (5.5,2);
\draw[thick][dashed] (5.5,2) -- (6,2);
\draw[thick] (.5,1) -- (6.5,1);
\draw[thick][dashed] (6.5,1) -- (7,1);
\draw[thick] (2.5,0) -- (7.5,0);
\draw[thick][dashed] (7.5,0) -- (8,0);

\draw[thick] (-.5,2) -- (.5,1);
\draw[thick] (1.5,2) -- (3.5,0);
\draw[thick] (3.5,2) -- (5.5,0);
\draw[thick] (5.5,2) -- (7.5,0);

\draw[thick] (0,1.5)--(.5,2);
\draw[thick] (1.5,1)--(2.5,2);
\draw[thick] (2.5,0)--(4.5,2);
\draw[thick] (4.5,0)--(6,1.5);
\draw[thick][dashed] (6,1.5)--(6.25,1.75);
\draw[thick] (6.5,0)--(7,.5);
\draw[thick][dashed] (7,.5)--(7.25,.75);

\end{tikzpicture}
\caption{Case 5}
\label{2m5}
\end{center}
\end{figure}

\begin{itemize}
\item For the left child ($\Sigma(A,B \cup \{v\})$):

\begin{itemize}
\item If \Thex~ $=3$, we add three matching vertices to $A$, and are left with a hexagon and a triangle that share an edge. By Lemma~\ref{lem:snowman}, we get two critical cells. For each of the cells, $|A|$ increases by 5. 
\item If \Thex~ $=4$, we add four matching vertices to $A$, and are left with a hexagon and a triangle that share an edge. By Lemma~\ref{lem:snowman}, we get two critical cells. For each of the cells, $|A|$ increases by 6. 
\item If \Thex~ $=5$, we add five matching vertices to $A$, and are left with a hexagon and a triangle that share an edge. By Lemma~\ref{lem:snowman}, we get two critical cells. For each of the cells, $|A|$ increases by 7. 
\item If \Thex~ $\geq 6$, we add three matching vertices to $A$ and decrease \Thex~ by 3, then repeat Case 5. 

\end{itemize}

\item For the right child ($\Sigma(A \cup \{v\},B \cup N(v))$):

\begin{itemize}
\item We only add $v$ to $A$, decrease \Thex~ by one, and go to Case 6.\\
\end{itemize}
\end{itemize}

\noindent {\bf Case 6:}

\begin{figure}[!htb] 
\begin{center}
\begin{tikzpicture}[scale = .85]

\filldraw (-.5,2) circle (0.1cm); 

\draw[thick] (-.5,2) -- (5.5,2);
\draw[thick][dashed] (5.5,2) -- (6,2);
\draw[thick] (-.5,1) -- (6.5,1);
\draw[thick][dashed] (6.5,1) -- (7,1);
\draw[thick] (.5,0) -- (7.5,0);
\draw[thick][dashed] (7.5,0) -- (8,0);

\draw[thick] (-.5,2) -- (1.5,0);
\draw[thick] (1.5,2) -- (3.5,0);
\draw[thick] (3.5,2) -- (5.5,0);
\draw[thick] (5.5,2) -- (7.5,0);

\draw[thick] (-.5,1)--(.5,2);
\draw[thick] (.5,0)--(2.5,2);
\draw[thick] (2.5,0)--(4.5,2);
\draw[thick] (4.5,0)--(6,1.5);
\draw[thick][dashed] (6,1.5)--(6.25,1.75);
\draw[thick] (6.5,0)--(7,.5);
\draw[thick][dashed] (7,.5)--(7.25,.75);

\end{tikzpicture}
\caption{Case 6}
\label{2m6}
\end{center}
\end{figure}

\begin{itemize}
\item For the left child ($\Sigma(A,B \cup \{v\})$):
\begin{itemize}
\item We do not add anything to $A$ or change \Thex. Go back to Case 3.

\end{itemize}

\item For the right child ($\Sigma(A \cup \{v\},B \cup N(v))$):

\begin{itemize}
\item If \Thex~ $= 2$, we get no critical cells. 
\item If \Thex~ $\geq 3$, we add three vertices to $A$, including $v$, decrease \Thex~ by two, then go back to Case 2.\\
\end{itemize}
\end{itemize}



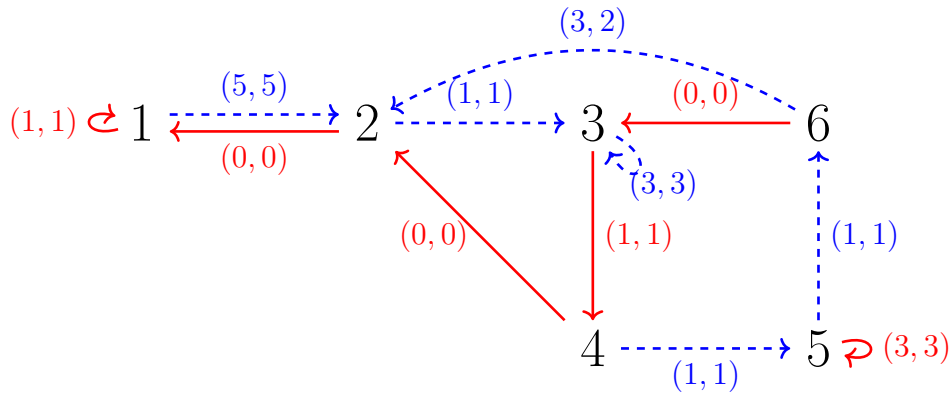
\begin{figure}[!hbt]
\begin{center}
\begin{tikzpicture}[scale = .75]


\node (1) at (0,0) {\LARGE $1$};
\node (2) at (4,0) {\LARGE $2$};
\node (3) at (8,0) {\LARGE $3$};
\node (6) at (12,0) {\LARGE $6$};

\node (4) at (8,-4) {\LARGE $4$};
\node (5) at (12,-4) {\LARGE $5$};

\draw [->, line width = 1pt, dashed, blue] (0.5,0.15) -- node [above] {$(5,5)$} (3.5, 0.15) ; 
\draw [->, line width = 1pt, red] (3.5,-0.15) -- node [below] {$(0,0)$} (0.5, -0.15); 

\draw [->, line width = 1pt, dashed, blue] (4.5,0) --node [above] {$(1,1)$} (7.5, 0);

\draw [->, line width = 1pt, red] (11.5,0) -- node [above] {$(0,0)$} (8.5, 0);

\draw [->, line width = 1pt, red] (8,-0.5) --node [right] {$(1,1)$} (8, -3.5);

\draw [->, line width = 1pt, dashed, blue] (12,-3.5) -- node [right] {$(1,1)$} (12, -0.5);

\draw [->, line width = 1pt, dashed, blue] (8.5,-4) -- node [below] {$(1,1)$} (11.5, -4);

\draw [->, line width = 1pt, red] (7.5,-3.5) -- node [left] {$(0,0)$} (4.5, -0.5);

\draw[->, line width = 1pt, dashed, blue] (6) to [bend right] node [above] {$(3,2)$} (2);

\draw[->, line width = 1pt, red] (1) edge [loop left] node {$(1,1)$} (1);

\draw (3) edge [->, line width = 1pt, out=330,in=300,looseness=8, dashed, blue] (3);

\draw[->, line width = 1pt, red] (5) edge [loop right] node {$(3,3)$} (5);

\node (3) at (9.25,-1.1) {\textcolor{blue}{$(3,3)$}};

\end{tikzpicture}
\caption{
In the proof of Theorem \ref{2by1bym}, 
there are six possible configurations, labeled 1 through 6 above. Choosing the appropriate splitting vertex produces two leaves: a left child and a right child. After split preparing each child, we arrive at new configurations.
The blue dashed lines are the right children while the red solid lines are the left children. The pair of values at each line are the increase to $|A|$ and the decrease to \Thex, respectively, from one configuration to the next. Note that the only time that they are not equal is when we go from 6 to 2. Here, we add 3 vertices to $A$ but only remove 2 hexagons.}
\label{fig:MTACases}
\end{center}
\end{figure}

\noindent \textbf{Analysis:} Figure~\ref{fig:MTACases} will aid in the analysis. Observe that in every case, $|A|$ increases at least as much as \Thex~decreases. Furthermore, whenever a critical cell is produced, $|A|$ increases  strictly more than \Thex~ decreases. Since \Thex~ $=2t$ at the start, the size of every critical cell is at least $2t+1$ and it follows that $d^{\min} = 2t$.

To find $d^{\max}$, we observe that we only increase $|A|$ more than we decrease \Thex~ without creating a critical cell in the right child of Case 6. The unique cycle containing this child is also made up of the right child of Case 2, the left child of Case 3, the right child of Case 4, and the right child of Case 5. Going once through this cycle increases $|A|$ by 7 and decreases \Thex~ by 6. It follows that if we start from Case 2 and proceed through several steps of the algorithm without creating a critical cell, the most we can increase $|A|$ is $\frac{7}{6}$ of the amount we decrease \Thex. To get from Case 1 to Case 2, we have to first increase $|A|$ by 5 and decrease \Thex~ by 5. This means that if we start from Case 1 and proceed through several steps of the algorithm without creating a critical cell, we increase $|A|$ by strictly less than $\frac76$ of the amount we decrease \Thex. When we do create a critical cell, the difference between the amount we increase $|A|$ and the amount we decrease \Thex~ is at most 2. Thus, the largest critical cells must be smaller than $\frac{7}{6}(2t) +2 = \frac{7t}{3} + 2$, and $d^{\max} = \frac{7t}{3} + 1$. 
\end{proof}

\begin{remark}
Corollary \ref{1by1bym} and Theorem \ref{2by1bym} show that for $1\times 1\times t$ and $2\times 1 \times t$ honeycomb graphs, the connectivity of the matching complex is at least one less than the number of hexagons in the graph. However, this is not true in general. Using homology tools in Sage, we found that $H_9$ of the matching complex of a $3 \times 2 \times 2$ honeycomb is $\mathbb{Z} \times \mathbb{Z}$. Because $H_9$ is nontrivial, the honeycomb is at most 8-connected, although it contains 10 hexagons.
\end{remark}


\section{Matching Complexes of Trees}
\label{sec:trees}

In this section, we  determine the explicit homotopy type of matching complexes of particular types of trees. 
Marietti and Testa proved the following: 

\begin{theorem}\cite[Theorem 4.13]{Marietti_Testa_forests}
\label{thm:forest}
Let $G$ be a forest. Then $M(G)$ is either contractible or homotopy equivalent to a wedge of spheres. 
\end{theorem}

We build on this result by determining the explicit homotopy type for caterpillar graphs and proving a lower bound for the connectivity of perfect binary trees.

\subsection{Matching Complexes of Caterpillar Graphs}
\label{sec:caterpillar}

A \emph{caterpillar graph} is a tree in which every vertex is on a central path or only one edge away from the path (see Figure~\ref{fig:gencase}). By Theorem~\ref{thm:forest}, we know that the matching complex of caterpillar graphs is a wedge summand of spheres (or contractible). 
In \cite{Marietti_Testa_forests}, Marietta and Testa note that one can compute the number of spheres in each dimension recursively using \cite[Proposition 3.3]{Marietti_Testa_forests}. We found that in special cases nice formulas arise.

We give the explicit homotopy type for matching complexes of caterpillar graphs with each central vertex incident to at least one leg. Table~\ref{table:gencat} contains calculations for the number of spheres in each dimension for the homotopy type of these caterpillar graphs. While this table appears complex, the entries have a nice combinatorial interpretation.

To describe the combinatorial interpretation we define the following class of subsequences of $[n] := (1,2,\ldots,n)$, which will enumerate the number of spheres of each dimension.

\begin{definition}
Let $n, x \in \mathbb{N}$. Let
\begin{flushleft}
${\max} := \{(i_1,i_2,\ldots,i_x) \subseteq [n] \mid 1 \leq i_1 < i_2 <\cdots <i_x \leq n$ and after marking the selected $ x $ positions the remaining $ n-x  $ positions can be covered with $ \frac{n-x}{2} $ disjoint 2-blocks.\}
\end{flushleft}

When $n-x$ is odd, ${\max} = \emptyset$. When $n = 2k$ and $x = 0$, $A_0^{2k} = \{ \emptyset\}$, so $|A_0^{2k} | = 1.$ 
\end{definition}

\begin{example} (Examples of $A_{x}^{n}$)
\begin{itemize}
\item[(1)] When $x=n$, $A_n^n = \{(1,2,\ldots,n)\}$.
\item[(2)] 
In Figure \ref{fig:seqex1}, we illustrate two subsequences of $(1, 2, \ldots, 9)$: one that is an element of $A_{3}^{9}$ and one that is not. 
\end{itemize}
\end{example}

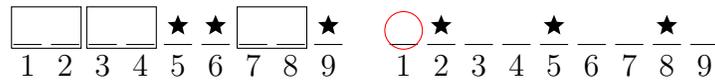
\begin{figure}[h!]

\begin{tikzpicture}[scale = 0.5]

\foreach \i in {1, 2, 3, 4, 5, 6, 7, 8, 9}{
\draw (0 + \i, 0) -- (0.75 + \i, 0);
\node[below] at (0.375 + \i, 0) {\i};}

\draw (0.95, -.1) -- (2.8, -.1)-- (2.8, 1) -- (0.95, 1) --(0.95, -.1) ;
\draw (2+ 0.95, -.1) -- (2+ 2.8, -.1)-- (2 +2.8, 1) -- (2+ 0.95, 1) --(2+ 0.95, -.1) ;
\draw (6+ 0.95, -.1) -- (6+ 2.8, -.1)-- (6 +2.8, 1) -- (6+ 0.95, 1) --(6+ 0.95, -.1) ;

 \draw (5.375, .5) node[star,star points = 5, star point ratio=2.25, fill=black, inner sep=1.3pt]  {};
  \draw (6.375, .5) node[star,star points = 5, star point ratio=2.25, fill=black, inner sep=1.3pt]  {};
    \draw (9.375, .5) node[star,star points = 5, star point ratio=2.25, fill=black, inner sep=1.3pt]  {};

   \foreach \i in {1, 2, 3, 4, 5, 6, 7, 8, 9}{
\draw (10 + 0 + \i, 0) -- (10 +0.75 + \i, 0);
\node[below] at (10+0.375 + \i, 0) {\i};} 
    \draw (12.375, .5) node[star,star points = 5, star point ratio=2.25, fill=black, inner sep=1.3pt]  {};
    \draw (15.375, .5) node[star,star points = 5, star point ratio=2.25, fill=black, inner sep=1.3pt]  {};
      \draw (18.375, .5) node[star,star points = 5, star point ratio=2.25, fill=black, inner sep=1.3pt]  {};

\draw[red] (11.375,0.375) circle (.5cm); 
    
\end{tikzpicture}

\caption{$(5,6,9) \in A_{3}^{9}$ but $(2, 5, 8) \notin A_{3}^{9}$}\label{fig:seqex1}
\end{figure}

\begin{remark}
When $n$ and $x$ have the same parity,
$$|A_x^n| = \binom{\frac{n+x}{2}}{\frac{n-x}{2}}.$$

\end{remark}

\begin{definition}
For any choice of $n$ nonnegative numbers $t_1,t_2,\ldots,t_n \in \mathbb{N}$ and $x \in \mathbb{N}$ such that $n-x$ is even, define the sum
\[
M_x^n = \sum\limits_{(i_1,i_2,\ldots,i_x) \in {\max}}t_{i_1} t_{i_2} \cdots  t_{i_x}
\]
\end{definition}

\begin{example}(Examples of $M_x^n$).

\begin{enumerate}
\item When $x = n$, $A_n^n = \{(1,2,3,\dots,n)\}$, so $M_n^n = t_1 t_2 \cdots  t_n$.

\item When $n$ is odd and $x =1$, $A_1^{2k+1} = \{(1), (3), (5), \ldots, (2k+1)\}$, so $M_1^{2k + 1} = t_1 + t_3 + t_5 + \cdots + t_{2k+1}$.

\item For an explicit example, consider when $n=4$ and $x=2$. Then, $A_2^4 = \{(1,2), (1,4), (3,4)\}$ and $M_2^4 = t_1t_2 + t_1t_4 + t_3t_4$.
\end{enumerate}

\end{example}

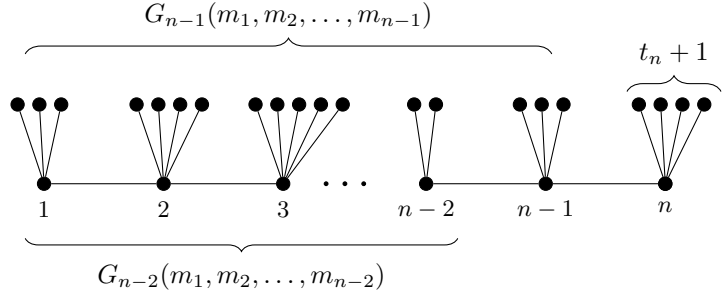
\begin{figure}

\begin{tikzpicture}[thin,
  fsnode/.style={draw, circle, fill = black, minimum size = 5 pt, inner sep=0pt},
  ssnode/.style={draw, circle, fill=black, minimum size = 3 pt, inner sep=0pt},
  every fit/.style={ellipse,draw,inner sep=-2pt,text width=1.5cm}
]

\begin{scope}[start chain=going right,node distance=14mm]
\foreach \i in {1,2,3}
  \node[fsnode,on chain,label=below:\scriptsize{$\i$}] (f\i) {};
\end{scope}

\begin{scope}[start chain=going right,node distance=1mm,yshift = 30pt, xshift = -10pt]
\foreach \i in {1,2,3}
  \node[fsnode,on chain] (g\i) {};
\end{scope}

\begin{scope}[start chain=going right,node distance=1mm,yshift = 30pt, xshift = 35pt]
\foreach \i in {1,2,3,4}
  \node[fsnode,on chain] (h\i) {};
\end{scope}

\begin{scope}[start chain=going right,node distance=1mm,yshift = 30pt, xshift = 80pt]
\foreach \i in {1,2,3, 4, 5}
  \node[fsnode,on chain] (n\i) {};
\end{scope}

\begin{scope}[start chain=going right,node distance=1mm,yshift = 30pt, xshift = 140pt]
\foreach \i in {1,2}
  \node[fsnode,on chain] (p\i) {};
\end{scope}

\begin{scope}[start chain=going right,node distance=1mm,yshift = 30pt, xshift = 180pt]
\foreach \i in {1,2,3}
  \node[fsnode,on chain] (k\i) {};
\end{scope}

\begin{scope}[start chain=going right,node distance=1mm,yshift = 30pt, xshift = 225pt]
\foreach \i in {1,2,3, 4}
  \node[fsnode,on chain] (m\i) {};
\end{scope}

\begin{scope}[start chain=going left,node distance=14mm, xshift = 235pt]
\foreach \i in {0, 1, 2}
  \node[fsnode,on chain] (s\i) {};
\end{scope}

\begin{scope}[start chain=going left,node distance=14mm, xshift = 235pt]
\foreach \i in {0}
  \node[fsnode,on chain, label=below:\scriptsize{$n$}] (s\i) {};
\end{scope}

\node [below = of s2, label=above:\scriptsize{$n-2$}, yshift = 14pt] {};

\node [below = of s1, label=above:\scriptsize{$n-1$}, yshift = 14pt] {};

\filldraw (3.75, 0) circle (0.025cm);
\filldraw (4.25, 0) circle (0.025cm);
\filldraw (4, 0) circle (0.025cm);

\foreach \i in {1,2,3}
\draw (f1) -- (g\i);
\foreach \i in {1,2,3, 4}
\draw (f2) -- (h\i);
\foreach \i in {1,2,3, 4, 5}
\draw (f3) -- (n\i);

\foreach \i in {1,2}
\draw (s2) -- (p\i);

\foreach \i in {1,2,3}
\draw (s1) -- (k\i);

\foreach \i in {1,2,3,4}
\draw (s0) -- (m\i);

\draw (f1) -- (f2);
\draw (f2) -- (f3);
\draw (s0) -- (s1);
\draw (s1) -- (s2);

\draw [decorate,decoration={brace,amplitude=5pt,mirror,raise=4ex}]
  (-.25,0) -- (5.5,0) node[midway,yshift=-3em]{\footnotesize{$G_{n-2}(m_1,m_2,\ldots, m_{n-2})$}};
  
  \draw [decorate,decoration={brace,amplitude=6pt,raise=4ex}]
    (-.25,1) -- (6.75,1) node[midway,yshift=+3em]{\footnotesize{$G_{n-1}(m_1,m_2,\ldots, m_{n-1})$}};

      \draw [decorate,decoration={brace,amplitude=6pt,raise=4ex}]
    (7.75,0.5) -- (9,0.5) node[midway,yshift=+3em]{\footnotesize{$t_{n}+1$}};
    
  
\end{tikzpicture}

\caption{A caterpillar graph of length $n$.}\label{fig:gencase}
\end{figure}

\begin{lemma}
\label{lem:Mrecursion}

The polynomials $M_{x}^{n}(t_1, \ldots, t_n)$ satisfy the following relations for $1 \leq  \ell \leq k$:
\begin{enumerate}
\item[(a)] $M_{2\ell+1}^{2k+1} = M_{2\ell}^{2k}t_{2k+1} + M_{2\ell+1}^{2k -1}$
\item[(b)] $M_{2\ell}^{2k+2}= M_{2\ell-1}^{2k+1}t_{2k+2} + M_{2\ell}^{2k}$
\end{enumerate}
\end{lemma}

\begin{proof}
We prove (a).
By definition, 
$M_{2\ell+1}^{2k+1}
= \sum\limits_{(i_1,\ldots, i_{2 \ell + 1}) \in A_{2\ell+1}^{2k+1}}t_{i_1}  t_{i_2} \cdots t_{i_{2 \ell + 1}}$.
We consider two cases.

\begin{enumerate}
\item If $i_{2\ell+1} = 2k+1$ then from the first $2k$ positions we choose $2\ell$ positions that we mark so that the unmarked positions can be covered with $\frac{2k-2\ell}{2}$ disjoint 2-blocks. The terms in $M_{2\ell +1}^{2k+1}$ that correspond to these subsequences in $A_{2\ell +1}^{2k+1}$ are exactly the terms in $M_{2\ell}^{2k}t_{2k+1}$.

\item If $i_{2\ell+1} < 2k+1$, the last position is not marked so positions $2k$ and $2k+1$ must be covered with one 2-block. We choose $2\ell +1$ positions to mark among the first $2k-1$ positions, respecting the rule. The sum corresponding to these subsequences is $M_{2\ell +1}^{2k-1}$.
\end{enumerate}
Combining (1) and (2) proves that $M_{2\ell+1}^{2k+1} = M_{2\ell}^{2k}t_{2k+1} + M_{2\ell+1}^{2k -1}$.  The proof of (b) is analogous. 
\end{proof}

\begin{theorem}
\label{thm:gencat}
Let $G_n = G_n(m_1, \ldots, m_n)$ be a caterpillar graph with $m_i$ legs, $m_i \geq 1$, at each vertex $i$ on the central path of length $n \geq 1$. Then
the homotopy type of $M(G_n)$ is given by:  

\begin{equation*}
    M(G_n) \simeq \begin{cases}
        \bigvee\limits_{\ell=0}^k \bigvee\limits_{M_{2\ell}^{2k}}S^{k-1+\ell} & \text{if } n = 2k \\
        \bigvee\limits_{\ell=0}^{k} \bigvee\limits_{M_{2\ell +1}^{2k+1}}S^{k+\ell} & \text{if } n = 2k+1
        \end{cases}
\end{equation*} 
where all sums $M_x^n$ are $M_x^n(t_1,t_2,\ldots,t_n)$ with $t_i := m_i-1$. 
\end{theorem}

\begin{proof}
We proceed by induction, using $G_1$ and $G_2$ as our base cases. When $n = 1, k = 0$ and $M_1^1 = t_1$. Notice that $M(G_1) = [m_1] \simeq \bigvee\limits_{m_1-1}S^0 = \bigvee\limits_{t_1}S^0$ as expected. 
When $n = 2, k=1,$ and $M_2^{2}(t_1,t_2) = t_1t_2$. The matching complex of $G_2$ is a 1-dimensional complex consisting of a disjoint point and a bipartite graph with the shores of $t_1 + 1$ and $t_2 + 1$ vertices, respectively. Therefore, $M(G_2) \simeq S^0 \vee \left[\bigvee\limits_{t_1t_2} S^1\right]$.

Assume the theorem holds for $G_1,G_2,\ldots,G_{2(k-1)+1}, G_{2k}$. We show that it holds for $G_{2k+1}$ and $G_{2k+2}$ for $k \geq 1$. Applying Lemma~\ref{lem:m}, we find that
\[
M(G_n) \simeq \left[\bigvee\limits_{t_n}S(M(G_{n-1}))\right] \vee S(M(G_{n-2}))
\]
for $n \geq 3$ as shown in Figure \ref{fig:gencase}.
Next, we use the properties of suspensions of wedges of spheres from Remark~\ref{rmk:topology} to see that:

\begin{eqnarray*}
M(G_{2k+1}) &\simeq& \left[ \bigvee\limits_{t_{2k+1}}S(M(G_{2k})) \right] \vee S(M(G_{2(k-1) +1}))\\
&\simeq& \left[ \bigvee\limits_{t_{2k+1}}S \left[ \bigvee\limits_{\ell=0}^{k} \bigvee\limits_{M_{2\ell}^{2k}}S^{k-1 + \ell} \right] \right]\vee S \left[ \bigvee_{\ell=0}^{k-1} \bigvee\limits_{M_{2\ell +1}^{2(k-1)+1}}S^{k-1 + \ell} \right]\\
&\simeq& \left[ \bigvee\limits_{t_{2k+1}}\bigvee\limits_{\ell=0}^{k} \bigvee\limits_{M_{2\ell}^{2k}}S(S^{k-1 + \ell}) \right] \vee \left[ \bigvee_{\ell=0}^{k-1} \bigvee\limits_{M_{2\ell +1}^{2(k-1)+1}}S(S^{k-1 + \ell})\right]\\
&=& \left[\bigvee\limits_{\ell=0}^{k-1} \bigvee\limits_{M_{2\ell}^{2k}t_{2k+1} + M_{2\ell+1}^{2(k-1) +1}} S^{k+\ell} \right] \vee \left[\bigvee\limits_{M_{2k}^{2k}t_{2k+1}}S^{2k} \right]
\end{eqnarray*}
 Since by Lemma \ref{lem:Mrecursion}, $M_{2\ell}^{2k}t_{2k+1} + M_{2\ell+1}^{2k -1} = M_{2\ell+1}^{2k+1}$, and
 $M_{2k}^{2k}t_{2k+1} = t_1t_2 \cdots t_{2k}t_{2k+1} = M_{2k+1}^{2k+1}$, the statement holds for $M(G_{2k+1})$.

For $n=2k+2$ the argument is similar. 
which completes the proof. 

\end{proof}

\begin{remark}
After our original preprint was posted, Singh~\cite{singh_forest} generalized Theorem~\ref{thm:gencat} to determine the homotopy type of bounded degree complexes of the graphs $G_n$. There is also an alternative proof of Theorem \ref{thm:gencat} 
that uses inflated simplicial complexes and Theorem 6.2 from \cite{Bjorner2}. This proof will appear in a forthcoming note by the first author. 



\end{remark}

In the special case where the caterpillar graph is {\em perfect}, the homotopy type is especially nice. 

\begin{definition}
A \emph{perfect $m$-caterpillar of length $n$} is a caterpillar graph with $m$ legs at each vertex on the central path of $n$ vertices (see 
\end{definition}

\begin{corollary} \label{thm:1cat}
For $m \geq 2$, let $G_n^p$ be a perfect $m$-caterpillar graph of length $n \geq 1$. Then the homotopy type of $M(G_n^p)$ is given by: 
\begin{equation*}
    M(G_n^p) \simeq 
    \begin{cases}
        \bigvee\limits_{t=0}^{k} \bigvee\limits_{\alpha_t} S^{k-1+t} & \text{if } n = 2k \\
        \bigvee\limits_{t=0}^k \bigvee\limits_{\beta_t} S^{k+t} & \text{if } n = 2k+1
       \end{cases}
\end{equation*} 
where $\alpha_t = \binom{k+t}{k-t}(m-1)^{2t}$ and $\beta_t = \binom{k+1+t}{k-t}(m-1)^{2t+1}$.
\end{corollary}

Notice that the number of spheres of dimension $d$ in the homotopy type of perfect 2-caterpillars are counted by binomial coefficients.

\begin{corollary}
\label{rem:onecat}
The homotopy type of a perfect $1$-caterpillar graph of length $n \geq 1$ is
\begin{equation*}
    M(G_n^p) \simeq 
    \begin{cases}
        S^{k-1} & \text{if } n = 2k \\
       \text{pt} & \text{if } n = 2k+1
       \end{cases}
\end{equation*} 
\end{corollary}

\subsection{The homotopy type for general caterpillar graphs}

Thus far, we have determined the explicit homotopy type of caterpillar graphs $G_n = G_n(m_1,\ldots,m_n)$ in the case where $m_i \geq 1$ for all $i$.
For arbitrary caterpillar graphs that may have vertices on the central path without any legs, there is no obvious formula that gives the number of spheres in each dimension. 
In this section, we provide a general procedure that inductively constructs the explicit homotopy type of $M(G_n(m_1,\ldots,m_n))$ where $m_i \geq 0$. 
More precisely, we have the following theorem.

\begin{theorem}
\label{thm:reallygencat}
Let $G_n = G_n(m_1,\ldots,m_n)$ be a caterpillar graph on a central path of $n$ vertices such that the $i$th vertex is a $m_i$-leg, where $n \geq 3$, $m_i \geq 0$ and $m_1 \geq 1$. Let $t_i = m_i -1$.  Let $A_{n, d}$ denote the number of spheres of dimension $d$ in the homotopy type of $M(G_{n})$. Then 

$$A_{n+1,d} = 
\begin{cases}
t_{n+1} A_{n, d-1} + A_{n-1, d-1} & \text{if } m_{n+1} \geq 1 \\
A_{n, d} + A_{n-1, d-1} & \text{if } m_{n+1} = 0 \text{ and } m_n \geq 1 \\
A_{n-2,d-1} & \text{if } m_{n+1} = m_{n} = 0
\end{cases}.
$$
\end{theorem}


\begin{proof}
In each case, the recursion follows from Lemma~\ref{lem:m}.
\end{proof}

Using Theorem~\ref{thm:reallygencat} and the following base cases, we
can obtain the homotopy type of any caterpillar graph:
\begin{align*} 
&M(G_1(m_1)) \simeq \underset{t_1}{\vee} S^0, & M(G_2(m_1,0)) = M(G_1(m_1+1)) \simeq \underset{t_1+1}{\vee} S^0,\\
&M(G_2(m_1,m_2)) \simeq S^0 \vee (\underset{t_1t_2}{\vee}S^1), &M(G_3(m_1,0,m_3)) \simeq \underset{t_1 + t_3 + t_1t_3}{\vee} S^1,
\end{align*}
where recall that $t_i = m_i - 1$.

Although there is no known nice formula for general caterpillars,
for some families of caterpillar graphs with a pattern of where the zero leg vertices occur, the number of spheres in each dimension does have a nice combinatorial description. 

For example, the matching complexes of caterpillar graphs with the property that every other vertex on the central path has zero legs (see Figure \ref{fig:nicewedge}) are homotopy equivalent to a wedge of spheres in a single dimension. Before we prove this using Theorem~\ref{thm:reallygencat}, we will define the following sum that enumerates the number of spheres in the wedge.

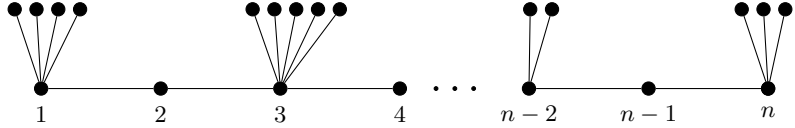
\begin{figure}[htb]

\begin{tikzpicture}[thin,
  fsnode/.style={draw, circle, fill = black, minimum size = 5 pt, inner sep=0pt},
  ssnode/.style={draw, circle, fill=black, minimum size = 3 pt, inner sep=0pt},
  every fit/.style={ellipse,draw,inner sep=-2pt,text width=1.5cm}
]

\begin{scope}[start chain=going right,node distance=14mm]
\foreach \i in {1,2,3, 4}
  \node[fsnode,on chain,label=below:\scriptsize{$\i$}] (f\i) {};
\end{scope}

\begin{scope}[start chain=going right,node distance=1mm,yshift = 30pt, xshift = -10pt]
\foreach \i in {1,2,3, 4}
  \node[fsnode,on chain] (g\i) {};
\end{scope}


\begin{scope}[start chain=going right,node distance=1mm,yshift = 30pt, xshift = 80pt]
\foreach \i in {1,2,3, 4, 5}
  \node[fsnode,on chain] (n\i) {};
\end{scope}

\begin{scope}[start chain=going right,node distance=1mm,yshift = 30pt, xshift = 185pt]
\foreach \i in {1,2}
  \node[fsnode,on chain] (p\i) {};
\end{scope}


\begin{scope}[start chain=going right,node distance=1mm,yshift = 30pt, xshift = 265pt]
\foreach \i in {1,2,3}
  \node[fsnode,on chain] (m\i) {};
\end{scope}

\begin{scope}[start chain=going left,node distance=14mm, xshift = 275pt]
\foreach \i in {0, 1, 2}
  \node[fsnode,on chain] (s\i) {};
\end{scope}

\begin{scope}[start chain=going left,node distance=14mm, xshift = 275pt]
\foreach \i in {0}
  \node[fsnode,on chain, label=below:\scriptsize{$n$}] (s\i) {};
\end{scope}

\node [below = of s2, label=above:\scriptsize{$n-2$}, yshift = 14pt] {};

\node [below = of s1, label=above:\scriptsize{$n-1$}, yshift = 14pt] {};

\filldraw (5.75, 0) circle (0.025cm);
\filldraw (5.25, 0) circle (0.025cm);
\filldraw (5.5, 0) circle (0.025cm);

\foreach \i in {1,2,3,4}
\draw (f1) -- (g\i);
\foreach \i in {1,2,3, 4, 5}
\draw (f3) -- (n\i);

\foreach \i in {1,2}
\draw (s2) -- (p\i);


\foreach \i in {1,2,3}
\draw (s0) -- (m\i);

\draw (f1) -- (f2);
\draw (f2) -- (f3) -- (f4);
\draw (s0) -- (s1);
\draw (s1) -- (s2);

  

    
  
\end{tikzpicture}

\caption{A caterpillar graph with the property that every other vertex has zero legs.}\label{fig:nicewedge}
\end{figure}

\begin{definition}
For $k \geq 1$, let $L_{k}(a_1, \ldots, a_k)$ denote the following sum:
\[
L_{k}(a_1, \ldots, a_k) = 
\sum\limits_{\substack{\ell=1, 2, \ldots, k \\ 1 \leq i_1 < \cdots < i_{\ell} \leq k }} (i_2 - i_1)(i_3 - i_2) \cdots (i_{\ell} - i_{\ell -1}) a_{i_{1}} a_{i_{2}} \cdots a_{i_{\ell}}
\]
\end{definition}

\begin{example}(Examples of $L_{k}(a_1, \ldots, a_k)$).

\begin{itemize}
\item[(1)] First consider when $k =1$. The summand is the single term $a_1$, corresponding to $i_1 = 1$.
\item[(2)] When $k = 2$ the summand has three terms. For $\ell =1$ we have $a_1$ corresponding to $i_1 = 1$ and $a_2$ corresponding to $i_1 = 2$. For $\ell = 2$ we have $a_1a_2$ corresponding to $(i_1, i_2) = (1, 2)$. So the resulting sum is $L_{2}(a_1, a_2) = a_1 + a_2 + (2-1)a_1 a_2 = a_1 + a_2 + a_1 a_2$.

\item[(3)] We provide two more examples:
\begin{itemize}
\item[(i)]When $k=3,~ L_{3}(a_1, a_2, a_3) = a_1 + a_2 + a_3 + a_1 a_2 + 2 a_1 a_3 + a_2 a_3 + a_1 a_2 a_3$ 
\item[(ii)] When $k=4,~ L_{4}(a_1, a_2, a_3, a_4) = a_1 + a_2 + a_3 + a_4 + a_1a_2 + 2a_1a_3 + 3a_1a_4 + a_2a_3 + 2a_2a_4 + a_3a_4 + a_1a_2a_3 + 2a_1a_2a_4 + 2a_1a_3a_4 + a_2a_3a_4 + a_1a_2a_3a_4$.
\end{itemize}
\end{itemize}
\end{example}

The polynomials $L_{k}(a_1, \ldots, a_k)$ satisfy the following relation for $k \geq 2$, which is easy to verify: 

\begin{equation}
\label{eqn:Lrecursion}
L_{k}(a_1, \ldots, a_k) = 
a_k L_{k-1}(a_1, \ldots, a_{k-2}, (a_{k-1} + 1))
+ L_{k-1}(a_1, \ldots, a_{k-1}).
\end{equation}

Caterpillar graphs with the property that every other vertex on the central path has zero legs can always be written
$G^a_k := G_{2k-1}(m_1, 0, m_2, 0, \ldots, m_{k-1}, 0, m_{k})$. For if we have a caterpillar on $2k$ vertices of the form
$G_{2k}(m_1, 0, m_2, 0, \ldots, m_{k-1}, 0, m_{k}, 0)$, we could rewrite this as a caterpillar on $2k-1$ vertices:
$G_{2k-1}(m_1, 0, m_2, 0, \ldots, m_{k-1}, 0, m_{k} +1)$. So in a caterpillar graph denoted $G^a_k$, the number of vertices on the central path is always odd. 

With this notation, $m_i$ is the number of legs on the $(2i-1)$-th vertex of the central path, rather than the $i$th vertex of the central path. While this is not consistent with the notation in Theorem~\ref{thm:reallygencat}, it simplifies the statement and proof that follows. 

\begin{theorem}
\label{thm:m0}
Let $G^a_k = G_{2k-1}(m_1, 0, m_2, 0, \ldots, m_{k-1}, 0, m_{k})$ be a caterpillar graph such that vertices $2, 4, \ldots, 2k-2$ on the central path have degree 2, and vertices $1, 3, \ldots, 2k-1$ have $m_i$ legs ($m_i \geq 1$). Then, for $k \in \mathbb{N}$, the homotopy type of $M(G^a_k)$ is given by:  

\[
    M(G^a_k) \simeq
        \bigvee\limits_{L_{k}(a_1, a_2, \ldots, a_k)} S^{k-1}
        \]
where $a_{i} = m_{i} - 1$ for $i = 1, \ldots, k$.
\end{theorem}

\begin{proof}
We will prove the theorem by induction on $k$. When $k=1$, $M(G_1(m_1))$ is a discrete set of $m_1$ points, so $M(G_1) = \bigvee\limits_{a_1} S^0$ where $a_1 = m_1 - 1$ as desired.

Now assume the claim holds for $1, \ldots, k-1$ and consider $G_{k}^{a} = G_{2k-1}(m_1, 0, m_2, 0, \ldots, m_k)$. Since $m_{k} \geq 1$, by Theorem~\ref{thm:reallygencat}, 
$$A_{2k-1, d} = a_k A_{2k-2, d-1} + A_{2k-3, d-1}$$
where recall that $A_{2k -1, d}$ is the number of spheres of dimension $d$ in $G_{2k-1}(m_1, 0, m_2, 0, \ldots, m_k)$, $A_{2k-2, d-1}$ is the number of spheres of dimension $d-1$ in $G_{2k-2}(m_1, 0, m_2, 0, \ldots, m_{k-1}, 0)$ and
$A_{2k-3, d-1}$ is the number of spheres of dimension $d-1$ in $G_{2k-3}(m_1, 0, m_2, 0, \ldots, m_{k-1})$. Using the induction hypothesis and the fact that 
$$G_{2k-2}(m_1, 0, m_2, 0, \ldots, m_{k-1}, 0) = G_{2k-3}(m_1, 0, m_2, 0, \ldots, m_{k-1} + 1)$$ we see that $A_{2k-2, d-1} = L_{k-1}(a_1, a_2, \ldots, (a_{k-1}+ 1))$ and $A_{2k-3, d-1} = L_{k-1}(a_1, a_2, \ldots, a_{k-1})$ if $d = k-1$ and $A_{2k, d-1} = A_{2k-1, d-1} = 0$ otherwise. So by equation (\ref{eqn:Lrecursion}), $A_{2k-1, d} = L_{k}(a_1, a_2, \ldots, a_k)$ if $d = k-1$ and $0$ otherwise, as desired. 
\end{proof}

\subsection{Matching Complexes of Perfect Binary Trees}

We conclude our discussion of trees by presenting a connectivity bound for perfect binary trees. This result requires both Lemma 2.8 (a consequence of the Matching Tree Algorithm) and Lemma 2.13, which was key in our results for caterpillar graphs.

A perfect binary tree is
a rooted tree in which every non-leaf vertex has two children.
The {\em depth} of a node in a binary tree is the number of edges in the path from that node to the root and the {\em height} of a binary tree is the depth of a leaf. A perfect binary tree of height $h$ has $2^h$ leaves.

\begin{example}
\label{rem:SmallPerfectTrees}
Let $T_h$ denote a perfect binary tree of height $h$. $T_1$ is $P_3$, so $M(T_1) = S^0$. $M(T_2) \simeq S^1 \vee S^1 \vee S^1$. Using Lemma \ref{lem:m} and Theorem \ref{thm:gencat}, we find that $M(T_3) \simeq S^4 \bigvee 
\left[\underset{4}{\vee} S^{3} \right]$.
Using homology tools in Sage in conjunction with Theorem \ref{thm:forest}, $M(T_4) \simeq \bigvee\limits_{56} S^{8} \vee \bigvee\limits_{11} S^9$. 

\end{example}

For $h \geq 1$, define
$$
d_h = \sum\limits_{i = 0}^{ \lceil \frac{h}3 \rceil - 1} 2^{h-3i-1} -2.$$

\begin{theorem}
\label{prop:perfectbinary}

Let $T_h$ be a perfect binary tree with height $h \geq 1$. 
Then $M(T_h)$ is at least $d_h$-connected. 

\end{theorem} 
\begin{proof}
Since $d_1 = -1$ and $d_2 = 0$, the claim is clearly true for $h = 1, 2$.
Let $T_h$ be a perfect binary tree of height $h\geq 3$. The line graph of $T_h$ is made up of $h-1$ rows of triangles such that there are $2^{h-k}$ triangles in the 
$k^{th}$ row, where the rows are numbered $1$ to $h-1$ from bottom to top (see Figure~\ref{fig:perfect}). For $k\equiv 1 \bmod 3$, we label the lower left vertex of each triangle in the $k^{th}$ row $v^j_i$,  where $j = \frac{k+2}3$ and $i$ ranges from $1$ to $2^{h-k}$ moving from left to right.
If $h \equiv 1 \bmod 3$, we also label the top left vertex $v^j_1$ where $j = \frac{h+2}3$. In total, this gives us  $\sum\limits_{i = 0}^{L} 2^{h-3i-1}$ labeled vertices where $L = \lceil \frac{h}3 \rceil - 1$.

Because the distance between any pair of labeled vertices is at least three, they satisfy the conditions of Lemma~\ref{lem:buzzerbeater}. Thus, all critical cells have at least $\sum\limits_{i = 0}^{L} 2^{h-3i-1}$ elements and $M(T_h)$ must be at least $((\sum\limits_{i = 0}^{L} 2^{h-3i-1})-2)$-connected.
\end{proof}

A straightforward calculation gives
$$d_h = \cfrac{2^{h+2} -r_h}{7}, \ \  \textrm{where} \ \ r_h=
\begin{cases}
18, & h\equiv 0 \pmod 3 \\
15, & h \equiv 1 \pmod 3  \\
16, & h \equiv 2 \pmod 3
\end{cases}.
$$

The following proposition shows that the connectivity bound from Proposition~\ref{prop:perfectbinary} is tight.

\begin{theorem} \label{prop:ref} 
For all $h\ge 1$, the matching complex $M(T_h)$ is homotopy equivalent to a wedge of spheres which contains $S^{d_h +1}$. Consequently, the connectivity bound given in Proposition~\ref{prop:perfectbinary} is tight.
\end{theorem}
\begin{proof}
We know that  the claim  is true for $h = 1, 2, 3$ by Example~\ref{rem:SmallPerfectTrees}.
 Assume the claim holds for $T_{h-3}$, for some $h\ge 4$. We will prove it for $T_h$.

Let $T'_2$ denote a perfect binary tree of height 2, with one additional edge at the root, as shown in Figure~\ref{fig:smallgraph}. 
Observe that we can partition the leaves of $T_h$ so that they are contained in exactly $2^{n-2}$ copies of $T'_2$.

 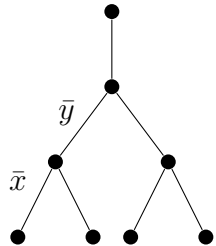
\begin{figure}[!htb]
\centering
\begin{tikzpicture}[level distance=10mm]
\node at (-.6, -1.35) {$\bar{y}$};
\node at (-1.25, -2.25) {$\bar{x}$};
\tikzstyle{every node}=[fill=black,circle,inner sep=2pt]
\tikzstyle{level 1}=[sibling distance=40mm,
set style={{every node}+=[fill=black]}]
\tikzstyle{level 2}=[sibling distance=15mm,
set style={{every node}+=[fill=black]}]
\tikzstyle{level 3}=[sibling distance=10mm,
set style={{every node}+=[fill=black]}]
\tikzstyle{level 4}=[sibling distance=4mm,
set style={{every node}+=[fill=black]}]
\node {}
child {node {}
child {node {}
child {node {}
}
child {node {}
}}
child {node {}
child {node {}
}
child {node {}
}}
};

\end{tikzpicture}
\caption{The subgraph $T_2'$.}
\label{fig:smallgraph}
\end{figure}

Let $\overline{x}$ and $\overline{y}$ be the edges shown in Figure~\ref{fig:smallgraph}. By Lemma~\ref{lem:m},

\begin{equation}\label{eq1}
M(T_h) \simeq S(M(\text{del}_{T_h}(\overline{x}))) \vee S(M(\text{del}_{T_h}(\overline{y}))).
\end{equation}

The subgraph $\text{del}_{T_h}(\overline{y})$ is the disjoint union of the subgraph $T_h \setminus T'_2$ (i.e. the subgraph of $T_h$ obtained by removing the edges of $T'_2$) 
 and 
the path $P_3$. 
 So by Remark~\ref{rem:disjoint},

\begin{equation}\label{eq2}
M(\text{del}_{T_h}(\overline{y})) \approx M(T_h \setminus T'_2) \ast M(P_3) \approx M(T_h \setminus T'_2) \ast S^0 \approx S(M(T_h \setminus T'_2)).
\end{equation}

From (\ref{eq1}) and (\ref{eq2}) we conclude that $M(T_h)$ has a a wedge-summand homotopy equivalent the double suspension  $S^2(M(T_h \setminus T'_2)).$ We repeat this procedure for each of the $2^{h-2}$ disjoint copies of the subgraph $T'_2$ in $T_h$, and track the same wedge-summand. After removing all copies of $T'_2$ from $T_h$, we obtain the tree $T_{h-3}$. Since suspension and wedge commute under homotopy eqivalence, we know that 
$M(T_h)$ contains a wedge summand homotopy equivalent to $2^{h-1}$-tuple suspension of $M(T_{h-3})$. 

By the inductive hypothesis, $M(T_{h-3})$ has a wedge summand homotopy equivalent to the sphere $S^{d_{h-3} +1}$, so $M(T_{h})$ has a wedge summand homotopy equivalent to the sphere of dimension:
$$ 2^{h-1} + d_{h-3} +1 =  2^{h-1} + \cfrac{2^{h-1} -r_{h-3}}{7} +1 = \cfrac{2^{h+2} -r_{h-3}}{7} +1 = d_h +1.$$
 In the last relation, we used $r_h= r_{h-3}$.
 This concludes the proof. 
\end{proof}
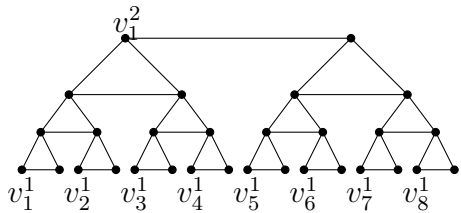
\begin{figure}
\begin{center}
\begin{tikzpicture}[scale = .5]

\foreach \i in {0, 1, 2, 3, 4, 5, 6, 7}{
\filldraw (-3 + 1.5*\i, 0) circle (0.1cm);
\filldraw (-2 + 1.5*\i, 0) circle (0.1cm);
\filldraw (-2.5 + 1.5*\i, 1) circle (0.1cm);

\draw (-3 + 1.5*\i, 0) -- (-2 + 1.5*\i, 0) -- (-2.5 +1.5*\i, 1) -- (-3 + 1.5*\i, 0);}

\foreach \i in {1, 2, 3, 4, 5, 6, 7,8}{
\node[below] at (-3 + 1.5*\i - 1.5, .1) {$v^1_{\i}$};
\node[below] at (-2 + 1.5*\i - 1.5 , .1) {};
\node[above] at (-2.5 + 1.5*\i - 1.5, 1) {};
}

\foreach \i in {0, 2, 4, 6}{
\draw (-2.5 +1.5*\i, 1) -- (-2.5 + 1.5 +1.5*\i, 1);}

\foreach \i in {0, 1, 2, 3}{
\filldraw (-1.75 + 3*\i, 2) circle (0.1cm);}

\foreach \i in {0, 1, 2, 3}{
\draw (-2.5 +1.5*2*\i, 1) --(-1.75 + 3*\i, 2);
}

\foreach \i in {0, 1, 2, 3}{
\draw (-2.5 +1.5+ 1.5*2*\i, 1) --(-1.75 + 3*\i, 2);
}

\draw (-1.75 + 0, 2) --(-1.75 + 3, 2) ;
\draw (-1.75 + 6, 2) --(-1.75 + 9, 2) ;

\filldraw (-.25, 3.5) circle (0.1cm);
\filldraw (-.25 + 6, 3.5) circle (0.1cm);

\draw (-.25, 3.5)--(-.25 + 6, 3.5);
\draw (-1.75, 2) --(-.25, 3.5) -- (-1.75+3, 2) ;

\draw (-1.75 + 6, 2) --(-.25 + 6, 3.5) --(-1.75 + 9, 2)  ;

\node at (-.2, 3.9) {$v^2_1$};

\end{tikzpicture}
\end{center}
\caption{The line graph of a perfect binary tree of height 4 $L(T_4)$.}
\label{fig:perfect}

\end{figure}

\section{Future Directions}
\label{sec:future}

In Section~\ref{polygon}, we explored the matching complexes of honeycomb graphs
and partially answered a question posed by Jonsson~\cite{Jakob} by presenting a weak lower bound for connectivity of an $r \times s \times t$ honeycomb graph and sharper bounds in the cases when $r = s= 1$ and $r =2$, $s =1$. The work to understand these complexes is far from done. 

We saw in Remark~\ref{2by1bym} that while a natural conjecture may be that the connectivity of the matching complex is at least one less than the number of hexagons, this does not turn out to be true. This raises the question: 
\begin{question}
What is the connectivity of $r \times s \times t$ honeycomb graphs, where $r \geq 3, s \geq 1,$ and $t \geq 1$?
\end{question}

In the introduction we also discussed Kekul\'e structures, which are perfect matchings of a honeycomb graph that have been studied in chemistry~\cite{klein_etal,hite_etal}. Perfect matchings of a honeycomb graph are the maximal dimensional faces of the matching complex, so we can consider the subcomplex induced by these faces, leading us to ask:

\begin{question}
What is the relationship between matching complexes of honeycomb graphs and Kekul\'e structures? 
\end{question}

In Section~\ref{sec:trees}, we turned our attention to trees and 
saw that the homotopy type of forests is either a point or a wedge of spheres. 
Consequently, acyclic graphs do not contain torsion.
From work by Shareshian and Wachs~\cite{Shareshian_Wachs} and Jonnson~\cite{Jakob}, we know that torsion appears in higher homology groups of the full matching complex $M(K_n)$ and the chessboard complex $M(K_{m,n})$. It would be interesting to
determine if there is torsion
in the higher homology groups of honeycomb graphs.
\begin{question}
Is there torsion in the matching complexes of $r \times s \times t$ honeycomb graphs?
\end{question}


\section*{Acknowledgements}

This work was partially completed during the 2019 Graduate Research Workshop in Combinatorics.  The workshop was partially funded by NSF  grants  1603823, 1604773 and  1604458,  ``Collaborative  Research:  Rocky  Mountain  -  Great Plains Graduate Research Workshops in Combinatorics,'' NSA grant H98230-18-1-0017, ``The 2018 and 2019 Rocky Mountain - Great Plains Graduate Research Workshops in Combinatorics,'' Simons Foundation Collaboration Grants \#316262 and \#426971 and grants from the Combinatorics Foundation and the Institute for Mathematics and its Applications. Additional funding was provided by Grant \#174034 of the Ministry of Education, Science and Technological Development of Serbia.

 We would like to thank Margaret Bayer and Bennet Goeckner for their insights and guidance on this project. Additionally, we would like to thank Benjamin Braun, Russ Woodroofe, Mario Marietti, and Damiano Testa for their helpful comments on this manuscript. Finally, we would like to thank the anonymous referee whose comments sketched a proof Theorem~\ref{prop:ref}.


\nocite{Michelle}
\nocite{Ben}

\newpage
\bibliographystyle{plain}
\bibliography{mc}

\addresseshere

\appendix

\newpage
\section{Computations for the homotopy type of trees}\label{sec:appendix}
\begin{minipage}{\linewidth}
\centering\begin{sideways}
\scalebox{1}{

\renewcommand{\arraystretch}{2}

\begin{tabular}{|p{3cm}|l|l|p{2cm}|p{3cm}|p{3.5cm}|p{2cm}|p{2cm}|c}
\hline
                                                                     \backslashbox{Tree}{Dim}               & 0     & 1           & 2                          & 3                                                     & 4                                                                                                                      & 5                                                                                                           & 6                       & 7 \\ \hline
$G_1(m_1)$                                                                          & $t_1$ &             &                            &                                                       &                                                                                                                        &                                                                                                             &                         &   \\ \hline
$G_2(m_1,m_2)$                                                                      & 1     & $t_1t_2$    &                            &                                                       &                                                                                                                        &                                                                                                             &                         &   \\ \hline
$G_3(m_1,m_2,m_3)$                                                                  & 0     & $t_1 + t_3$ & $t_1t_2t_3$                &                                                       &                                                                                                                        &                                                                                                             &                         &   \\ \hline
$G_4(m_1,\ldots,m_4)$                                                              & 0     & 1           & $t_1t_2 + t_1t_4 + t_3t_4$ & $t_1t_2t_3t_4$                                        &                                                                                                                        &                                                                                                             &                         &   \\ \hline
$G_5(m_1,\ldots, m_5)$                                                         & 0     & 0           & $t_1 + t_3 + t_5$          & $t_1t_2t_3 + t_1t_2t_5 + t_1t_4t_5 + t_3t_4t_5$       & $t_1t_2t_3t_4t_5$                                                                                                      &                                                                                                             &                         &   \\ \hline
$G_6( m_1, \ldots, m_6)$        & 0     & 0           & 1                          & $t_1t_2 + t_1t_4 + t_3t_4 + t_1t_6 + t_3t_6 + t_5t_6$ & $t_1t_2t_3t_4 + t_1t_2t_3t_6 + t_1 t_2t_5t_6 + t_1t_4t_5t_6 + t_3t_4t_5t_6$                                            & $t_1t_2t_3t_4t_5t_6$                                                                                        &                         &   \\ \hline
$G_7(m_1, \ldots, m_7)$ & 0     & 0           & 0                          & $t_1 + t_3 + t_5 + t_7$                               & $t_1t_2t_3 + t_1t_2t_5 + t_1t_4t_5 + t_3t_4t_5+ t_1t_2t_7 + t_1t_4t_7 + t_3t_4t_7 + t_1t_6t_7 + t_3t_6t_7 + t_5t_6t_7$ & $t_1t_2t_3t_4t_5 + t_1t_2t_3t_4t_7 + t_1t_2t_3t_6t_7 + t_1t_2t_5t_6t_7 + t_1t_4t_5t_6t_7 + t_3t_4t_5t_6t_7$ & $t_1t_2t_3t_4t_5t_6t_7$ &   \\ \hline
                                                                                    &       &             &                            &                                                       &                                                                                                                        &                                                                                                             &                         &  
\end{tabular}}
\end{sideways}
\captionof{table}{The number of spheres of dimension $d$ in the homotopy type of $M(G_{n}(m_1, \ldots, m_n))$, where $G_{n}(m_1, \ldots, m_n)$ is a caterpillar with $m_i=t_i+1$ legs at each vertex $i$ of the central path.}
\label{table:gencat}
\end{minipage}

\end{document}